\documentclass[12pt]{amsart}
\usepackage[colorlinks=true, pdfstartview=FitV, linkcolor=blue, citecolor=blue, urlcolor=blue, breaklinks=true]{hyperref}
\usepackage{amsmath,amsfonts,amssymb,amsthm,amscd,array,comment,
euscript,mathtools,etoolbox,latexsym,stmaryrd}
\usepackage[usenames,dvipsnames]{color}
\usepackage{paralist}
\usepackage[all]{xy}

\newtoggle{comments}
\newtoggle{details}
\newtoggle{detailsnote}

\toggletrue{comments}      

\togglefalse{details}      

\toggletrue{detailsnote}   

\newcommand\C{\mathbb{C}}
\newcommand\Z{\mathbb{Z}}

\newcommand\st{\mathrm{st}}
\newcommand\nat{\mathrm{nat}}

\newcommand\g{\mathfrak{g}}
\newcommand\fm{\mathfrak{m}}

\renewcommand\fm{\mathfrak{m}}
\newcommand\fh{\mathfrak{h}}
\newcommand\fn{\mathfrak{n}}

\newcommand\fb{\mathfrak{b}}
\newcommand\fsl{\mathfrak{sl}}
\newcommand\fpsl{\mathfrak{psl}}
\newcommand\fosp{\mathfrak{osp}}
\newcommand\fgl{\mathfrak{gl}}

\newcommand\fu{\mathfrak{u}}
\newcommand\fk{\mathfrak{k}}

\newcommand\fH{\mathfrak{H}}

\newcommand\D{\Delta}
\renewcommand\S{\Sigma}

\newcommand\cA{\mathcal{A}}
\newcommand\cB{\mathcal{B}}

\newcommand\cG{\mathcal{G}}
\newcommand\cH{\mathcal{H}}
\newcommand\cK{\mathcal{K}}

\newcommand\cM{\mathcal{M}}
\newcommand\cN{\mathcal{N}}
\newcommand\cP{\mathcal{P}}

\newcommand\cU{\mathcal{U}}

\newcommand\ba{\mathbf{a}}
\newcommand\bb{\mathbf{b}}
\newcommand\bc{\mathbf{c}}

\newcommand\bi{\mathbf{i}}
\newcommand\bm{\mathbf{m}}

\newcommand\bp{\mathbf{p}}

\newcommand\bU{\mathbf{U}}

\DeclareMathOperator{\Span}{Span}
\DeclareMathOperator{\LinSpan}{LinSpan}

\DeclareMathOperator{\Supp}{Supp} 

\DeclareMathOperator{\Ind}{Ind}

\DeclareMathOperator{\hgt}{ht}

\theoremstyle{plain}
\newtheorem{theo}{Theorem}[section]
\newtheorem*{theo*}{Theorem}
\newtheorem{prop}[theo]{Proposition}
\newtheorem{lem}[theo]{Lemma}
\newtheorem{cor}[theo]{Corollary}

\theoremstyle{definition}

\newtheorem*{rem*}{Remark}
\newtheorem{rem}[theo]{Remark}

\newtheorem{example}[theo]{Example}

\numberwithin{equation}{section}

\allowdisplaybreaks

\toggletrue{detailsnote}   

\iftoggle{comments}{%
  \newcommand{\comments}[1]{
    \begin{center}
      \parbox{6.5 in}{
        \color{red}
          {\footnotesize \textbf{Comments:} #1}
        \color{black}}
    \end{center}}
}{%
  \newcommand{\comments}[1]{}
}

\iftoggle{details}{%
  \newcommand{\details}[1]{
      \ \\
      \color{OliveGreen}
        \begin{footnotesize}
          \textbf{Details:} #1
        \end{footnotesize}
      \color{black}
      \\
  }
}{%
  \newcommand{\details}[1]{}
}

\begin{document}
%

\title[Highest weight  modules for affine Lie superalgebras]
{Highest weight  modules for affine Lie superalgebras}
\author{Lucas Calixto}
\address{Departamento de Matemática\\
		   Instituto de Ci\^encias Exatas\\
		   UFMG\\
		   Belo Horizonte, Minas Gerais, Brazil, 30.123-970}
		   \email{lhcalixto@ufmg.br}	
\author{Vyacheslav Futorny}
\address{IME-USP, Caixa Postal 66281, CEP 05315-970, S\~ao Paulo, Brazil}
\email{futorny@ime.usp.br}
\subjclass{Primary 17B67}
\date{}

\keywords{affine Lie superalgebras, highest weight modules}

\begin{abstract}
We describe Borel and parabolic subalgebras of affine Lie superalgebras and study the Verma type modules associated to such subalgebras. We give necessary and sufficient conditions under which these modules  are simple. 

\end{abstract}

\maketitle \thispagestyle{empty}



\medskip

\section{Introduction}
%
Affine Lie algebras and superalgebras play an important role in many areas of physics and mathematics, such as symmetries, string theory, and number theory (see for instanse \cite{KW94, FK02}), and their representation theory is an active area of research. A very important class of modules are the so called weight modules, which are by definition modules on which a Cartan subalgebra acts semisimply. Among them are the Verma type modules,  which are associated to Borel subalgebras, and the generalized Verma type modules which are associated to parabolic subalgebras. It is well known that their structure depends deeply on the subalgebras from which they are induced. For instance, the weight spaces of a Verma type module are finite-dimensional if and only if the Borel subalgebra comes from a system of simple roots. The structure of Verma and generalized Verma type modules have been extensively studied for affine Lie algebras (see for instance, \cite{Fut97, FS93, Fut94, Cox94, KW01}). In this paper we study these modules for affine Lie superalgebras. 

Let $\g=\g_{\bar 0}\oplus \g_{\bar 1}$ be one of the Lie superalgebras $\fgl(m,n)$, $\fsl(m,n)$ with $m\neq n$, $\fosp(m,n)$, $D(2,1;\alpha)$ ($a\neq 0,-1$), $G(3)$ or $F(4)$, and let $\cG$ be its associated non-twisted affine Lie superalgebra. It is well known that parabolic sets and partitions of root systems behave nicer in the non-super setting than their super counterparts. For instance, if $\g$ is a semisimple Lie algebra, then that all parabolic partitions of the root system of $\g$ are conjugate under the action of the Weyl group of $\g$. In the super setting however this is no longer the case. In \cite{DFG09} the authors compare two combinatorial definitions of parabolic sets of roots systems, and prove that these definitions are equivalent in the non-super setting, which is not always the case in the supper setting. It was shown in \cite{Ser11} that any two parabolic partitions of a root system of $\g$ can be obtained one from the other by even and odd reflections. In Section~\ref{sec:par.sets.finite-case} we use these reflections to describe the parabolic sets of the root system of $\g$ (Proposition~\ref{prop:par.set}). In Section~\ref{sec:par.sets.affine-case}, combining the approaches of \cite{JK89, Fut92} and \cite{Ser11} we are able to describe all the parabolic partitions/sets for root systems of $\cG$. In particular, this provides a description of all parabolic and Borel subalgebras of $\cG$ (see Propositions~\ref{prop:delta.notin.P_0},~\ref{prop:delta.in.P_0} and Theorem~\ref{thm:descr.of.G_0}). 

A parabolic subalgebra $\cP$ of $\cG$ (associated to a parabolic set $P$) induces a decomposition $\cG=\cU^-\oplus (\cG^0+\cH)\oplus \cU$ (see Theorem~\ref{thm:parabolic.structure}).  Any given $(\cG^0+\cH)$-module $N$ can be extended to a $\cP$-module with trivial action of $\cU$. Then we can construct the induced module
	\[
M_\cP(N)=\Ind^\cG_{\cP} N.
	\]
When $N$ is simple, such a module is called a \emph{generalized Verma type module}, following \cite{Fut97}. In this case,  $M_\cP(N)$  admits a unique maximal proper submodule, and thus, a unique simple quotient $L_\cP(N)$.  
 A simple module is \emph{parabolic induced} if it is isomorphic to  $L_\cP(N)$ for some $\cP \neq \cG$ and $N$, otherwise  the module is  \emph{cuspidal}. Finite-dimensional Lie superalgebras that admit  weight cuspidal modules with finite-dimensional weight spaces
 were classified by Dimitrov, Mathieu and Penkov in \cite{DMP04}. 
 Such simple subalgebras of  affine Lie superalgebras were classified by Futorny and Rao in \cite[Corollary~6.1]{EF09}. A list of such subalgebras is comprised by: simple Lie subalgebras of type $A$ or $C$; $\frak{osp}(m,2n)$, for $m=1,3,4,5,6$; and $D(2,1;a)$.  On the other hand,  it was shown in \cite{EF09}  that any simple weight module
  of non-zero level with finite-dimensional weight spaces is parabolic induced from some parabolic subalgebra $\cP\subseteq \cG$.

  In Section~\ref{sec:main.sec} we consider non-standard Verma type  (resp. generalized Verma type) modules associated to a given non-standard Borel (resp. parabolic) subalgebra $\cB$ (resp. $\cP$) of $\cG$. Our main result, Theorem~\ref{thm:main}, gives a necessary and sufficient condition under a generalized Verma type module $M_\cP(N)$ with a nonzero central charge is simple. As a consequence we obtain a simplicity criterion for a Verma type module  $M_{\cB}(\lambda)$ in Corollary \ref{cor-main} (see also  Corollary \ref{cor-empty}). This generalizes  similar results for  affine Lie algebras from \cite{FS93, Fut94, Cox94} to the case of non-twisted affine Lie superalgebras. Even though the results were expected, the proof required a careful treatment of the super situation.

\medskip

\iftoggle{detailsnote}{
\medskip

\paragraph{\textbf{Note on the arXiv version}} For the interested reader, the arXiv version of this paper includes hidden details of some straightforward computations and arguments that are omitted in the pdf file.  These details can be displayed by switching the \texttt{details} toggle to true in the tex file and recompiling.
}{}

%
\section{Preliminaries}
%


%

In this paper we assume that $\g=\g_{\bar 0}\oplus \g_{\bar 1}$ is one of the Lie superalgebras $\fgl(m,n)$, $\fsl(m,n)$ with $m\neq n$, $\fosp(m,n)$, $D(2,1;\alpha)$ ($a\neq 0,-1$), $G(3)$ or $F(4)$. In particular, $\g_{\bar 0}$ is a reductive Lie algebra, and $\g$ admits an even non-degenerate supersymetric invariant bilinear form $(\cdot | \cdot)$. Choose a Cartan subalgebra $\fh$ of $\g$ (i.e. a Cartan subalgebra of $\g_{\bar 0}$) and let $\g=\fh\oplus\left(\bigoplus_{\alpha\in \dot{\D}} \g_\alpha\right)$ be the root space decomposition of $\g$, where $\g_\alpha$ denotes the root space associated to $\alpha\in \dot{\D}\subseteq \fh^*$. Recall that every element of $\dot{\D}$ is either purely even or purely odd, meaning that, for any $\alpha\in \dot{\D}$, either $\g_\alpha\subseteq \g_{\bar 0}$, or $\g_\alpha\subseteq \g_{\bar 1}$. Then $\dot{\D}=\dot{\D}_{\bar 0}\cup \dot{\D}_{\bar 1}$, where $\dot{\D}_i$ denote the subset of all roots with parity $i\in \Z_{2}$ of $\dot{\D}$. For any subset $X\subseteq \dot{\D}$ we set $X_i=X\cap \dot{\D}_i$, for $i\in \Z_2$.

A linearly independent subset $\dot{\S}\subseteq \dot{\D}$ will be called a \emph{base} if one can find elements $x_\alpha\in \g_\alpha$, $y_\alpha\in \g_{-\alpha}$, such that $\{x_\alpha, \ y_\alpha\mid \alpha\in \dot{\S}\}\cup \fh$ generates $\g$, and $[x_\alpha, y_\beta]=0$, for $\alpha\neq \beta$. These relations imply that every root of $\g$ is a purely non-negative or a purely non-positive integer linear combination of elements in $\dot{\S}$. Such roots are called positive or negative, respectively, and we decompose $\dot{\D}=\dot{\D}^+(\dot{\S})\cup \dot{\D}^-(\dot{\S})$, where $\dot{\D}^+(\dot{\S})$ and $\dot{\D}^-(\dot{\S})$ denote the set of positive and negative roots, respectively. A root of $\dot{\D}^+(\dot{\S})$ is said to be \emph{simple} if it cannot be decomposed as a sum of two positive roots. With this terminology a base $\dot{\S}$ is a \emph{system of simple roots} of $\dot{\D}$ in the usual sense. By construction, contragredient Lie superalgebras admit a base. If we let $h_\alpha:=[x_\alpha, y_\alpha]$, then, for the case where $\g\neq \fgl(n,m)$, we have that the set $\{h_\alpha\mid \alpha\in \dot{\S}\}$ gives a basis of $\fh$. If $\g = \fgl(n,m)$, then we first define $h_0:=I_{n,m}$ (where $I_{n,m}$ denotes the identity matrix of $\fgl(n,m)$), and the set $\{h_\alpha\mid \alpha\in \dot{\S}\}\cup \{h_0\}$ forms a basis for $\fh$. We call $\alpha\in \S_{\bar 1}$ an \emph{isotropic root} if $\alpha(h_\alpha) = 0$. It is well known that $\g$ admits systems of simple roots with only one odd root. Such systems are called \emph{distinguished}. If $\g$ is a semisimple Lie algebra, then all systems of simple roots are conjugate under the action of the Weyl group of $\g$. In the super setting however this is no longer the case (see \cite{Ser11}).

\begin{rem}\label{rmk:fix.not.A(n,n).sl(n,n)}
The Lie superalgebras $\fsl(n,n)$ and $A(n,n)$ are isomorphic, respectively, to $[\fgl(n,n),\fgl(n,n)]$ and $[\fgl(n,n),\fgl(n,n)]/C$, where $C = \C h_0$ is the one-dimensional center of $\fsl(n,n)$. For these Lie superalgebras the set $\{h_\alpha\mid \alpha\in \dot{\S}\}$ is no longer linearly independent, and, as a consequence, we are not allowed to talk about systems of simple roots as we do for the other cases.
\end{rem}

Let now $\cG=\cG_{\bar 0}\oplus \cG_{\bar 1}$ denote the non-twisted affine Lie superalgebra associated to $\g$. Namely, if we let $L(\g)=\g\otimes \C[t,t^{-1}]$ denote the loop superalgebra associated to $\g$, then
	\[
\cG=L(\g)\oplus \C K\oplus \C d,
	\]
with the following commutation relations: for all $x,y\in \g$, $m,n\in \Z$
\begin{gather*}
[x(m), y(n)]=[x,y](m+n) + m\delta_{m,-n}(x|y)K, \\
[d, x(m)]=-mx(m),\quad [K,\cG]=0,
\end{gather*}
where for every $z\in \g$, $m\in \Z$, $z(m)$ stands for the element $z\otimes t^m\in L(\g)$. Recall that $(\cdot | \cdot)$ extends from $\g$ to an even non-degenerate supersymetric invariant bilinear form $(\cdot | \cdot)$ on $\cG$ by:
\begin{gather*}
(x(m)|y(n))=\delta_{m,-n}(x|y), \quad (L(\g)|\C K + \C d)=0, \\
(K|K)=(d|d)=0,\quad (K|d)=-1.
\end{gather*}

Let 
	\[
\cH=\fh\oplus \C K\oplus \C d
	\]
be the Cartan subalgebra of $\cG$. As in the finite dimensional case, the action of $\cH$ on $\cG$ provides a root space decomposition $\cG=\cH\oplus\left(\bigoplus_{\alpha\in \D} \cG_\alpha\right)$, 
where $\Delta=\{\alpha+n\delta\mid \alpha\in\dot{\Delta}\cup\{0\},\ n\in\Z\}\setminus\{0\}\subseteq \cH^*$, and $\cG_\alpha$ denotes the root space associated to the root $\alpha\in \D$. Similarly to the finite-dimensional case, we have $\D=\D_{\bar 0}\cup \D_{\bar 1}$, where $\D_i$ denote the set of roots of $\D$ of parity $i\in \Z_2$.

Finally, for a subset $R\subseteq \dot{\D}$ (resp. $\D$), we define 
	\[
\langle R\rangle= \Z_{\geq 0}R \cap \dot{\D}\quad \text{(resp.} \langle R\rangle= \Z_{\geq 0}R \cap \D).
	\]

\section{Parabolic partitions and parabolic sets}
 A subset $P$ of $\D$ is said to be \emph{additively closed} if $\alpha,\beta\in P$, $\alpha+\beta\in \D$ implies that $\alpha+\beta\in P$. An additively closed subset $P\subseteq \D$ is called a \emph{parabolic partition} (resp. \emph{parabolic set}) if $P\cup (-P)=\D$ and $P_0 = P\cap (-P)=\emptyset$ (resp. $P_0 = P\cap (-P)\neq \emptyset$). Parabolic partitions are also called sets of positive roots in the literature.

\subsection{Finite case}\label{sec:par.sets.finite-case}
Recall that we are assuming $\g=\g_{\bar 0}\oplus \g_{\bar 1}$ equal $\fgl(m,n)$, $\fsl(m,n)$ with $m\neq n$, $\fosp(m,n)$, $D(2,1;\alpha)$ ($a\neq 0,-1$), $G(3)$ or $F(4)$. In this section we consider parabolic sets for the system of roots $\dot{\D}$ of $\g$. It is well known that there is a bijection between parabolic partitions and systems of simple roots of $\dot{\D}$. Namely, for every partition $P$, we define a system of simple roots $\dot{\S}$ of $\dot{\D}$ to be the set of all roots of $P$ that cannot be decomposed as sum of two roots of $P$. Then $P=\dot{\D}^+(\dot{\S})$. On the other hand, for a given system of simple roots $\dot{\S}$ of $\dot{\D}$, it is clear that $\dot{\D}^+(\dot{\S})$ defines a parabolic partition of $\dot{\D}$.

 Let $\dot{\S}\subseteq \dot{\D}$ be a system of simple roots. Recall the definition of reflections for Lie superalgebras:
\begin{enumerate}
\item If $\alpha\in \dot{\S}$ is an odd isotropic root (i.e. $\alpha(h_\alpha)=0$), then we define the odd reflection  $r_\alpha : \dot{\S}\to \dot{\D}$ by:
\begin{gather*}
r_\alpha(\alpha)=-\alpha,\\
r_\alpha(\beta)=\beta\quad \text{if }\alpha\neq \beta\text{ and } \alpha(h_\beta)=\beta(h_\alpha)=0,\\
r_\alpha(\beta)=\beta+\alpha\quad \text{if either }\alpha(h_\beta)\neq 0\text{ or } \beta(h_\alpha)\neq 0;\\
\end{gather*}
\item If $\alpha\in \dot{\D}$ is even, then we define the even reflection $r_\alpha=\fh^*\to \fh^*$ as usual: 
	\[
r_\alpha(\mu)=\mu - \mu(h_\alpha)\alpha,\quad \text{ for all } \mu\in \fh^*;
	\]
\item If $\alpha\in \dot{\D}$ is a non-isotropic odd root (i.e. $2\alpha=\beta\in \dot{\D}_{\bar 0}$), then we define $r_\alpha:=r_\beta$.
\end{enumerate}
For all cases, we define $\dot{\S}_\alpha:= r(\dot{\S})=\{r_\alpha(\beta)\mid \beta\in \dot{\S}\}$. It was shown in \cite{Ser11} that if $\dot{\S}$ and $\dot{\S}'$ are two systems of simple roots, then one can be obtained from the other by a chain of even and odd reflections. In particular, any two parabolic partitions of $\dot{\D}$ can be obtained one from the other by a chain of even and odd reflections. In what follows we use these reflections to obtain all parabolic sets of $\dot{\D}$.

\begin{lem}\label{lem:pos.system.in.par.sets}
Let $P$ be a parabolic set of $\dot{\D}$. Then there exists a system of positive roots $\dot{\D}^+$ such that $\dot{\D}^+\subseteq P$.
\end{lem}
\begin{proof}
Let $\dot{\D}^+=\dot{\D}^+(\dot{\S})$ be a system of positive roots of $\dot{\D}$. Suppose $\dot{\D}^+$ is not contained in $P$. Therefore there exists $\alpha\in \dot{\S}\setminus P$ (in particular, $\alpha\in -P$ and $-\alpha\in P$). Then 
\begin{align*}
\dot{\D}^+(\dot{\S}_\alpha)\cap (-P) & = ((\dot{\D}^+(\dot{\S})\setminus \{\alpha, 2\alpha\})\cup \{-\alpha,-2\alpha\})\cap (-P) \\
& =((\dot{\D}^+(\dot{\S})\cap (-P))\setminus \{\alpha, 2\alpha\})\cup (\{-\alpha,-2\alpha\}\cap (-P)) \\
&= (\dot{\D}^+(\dot{\S})\cap (-P))\setminus \{\alpha, 2\alpha\},
\end{align*}
where the last equality follows from the fact that neither $-\alpha$ nor $-2\alpha$ can lie in $-P$. Thus $|\dot{\D}^+(\dot{\S}_\alpha)\cap (-P)|<|\dot{\D}^+(\dot{\S})\cap (-P)|$. Since $|\dot{\D}|<\infty$, we can continue to apply this argument to get a system of simple roots $\dot{\S}'$ such that $\dot{\D}^+(\dot{\S}')\cap (-P)=\emptyset$. In particular, $\dot{\D}^+(\dot{\S}') \subseteq P$.
\end{proof}

\begin{lem}\label{lem:par.set.construction}
Let $\dot{\S}$ be a set of simple roots of $\dot{\D}$, and $R\subseteq \dot{\S}$ be a non-empty set. Then the set $P=\dot{\D}^+(\dot{\S})\cup \langle -R\rangle$ is a parabolic set of $\dot{\D}$.
\end{lem}
\begin{proof}
It is clear that $\dot{\D}=P\cup(-P)$. Let $\alpha,\beta\in P$ such that $\alpha+\beta\in\dot{\D}$. If either $\alpha,\beta\in\dot{\D}^+(\dot{\S})$ or $\alpha,\beta\in \langle -R\rangle$, then $\alpha+\beta\in P$, since both $\dot{\D}^+(\dot{\S})$, and $\langle -R\rangle$, are additively closed. The only remaining case is when $\alpha\in\dot{\D}^+(\dot{\S})$ and $\beta\in \langle -R\rangle$. In this case, we have
	\[
\alpha=\sum_{\gamma\in\dot{\S}}k_\gamma\gamma,\quad  \beta=\sum_{\eta\in R}-k_\eta' \eta,
	\]
where $k_\gamma\geq 0$, for all $\gamma\in\dot{\S}$, and $k_\eta'\geq 0$, for all $\eta\in R$.
Hence
	\[
\alpha+\beta=\sum_{\gamma\in\dot{\S}\setminus R}k_\gamma\gamma+
\sum_{\eta\in R}\ell_\eta \eta,
	\]
where $\ell_\eta=k_\eta-k_\eta'$, for all $\eta\in R$. Now, since $\dot{\S}$ is linearly independent, and every root is either a non-negative, or a non-positive integer linear combination of elements in $\dot{\S}$, we have the following possibilities:
\begin{enumerate}
\item $k_\gamma>0$ for some $\gamma\in\dot{\S}\setminus R$. In this case $k_\gamma\geq 0$, for all $\gamma\in\dot{\S}\setminus R$, $\ell_\eta\geq 0$, for all $\eta\in R$, and hence $\alpha+\beta\in\dot{\D}^+(\dot{\S})$.
\item $k_\gamma=0$, for all $\gamma\in\dot{\S}\setminus R$. Then either $\ell_\eta\geq 0$ for all $\eta\in R$, and $\alpha+\beta\in\dot{\D}^+(\dot{\S})$, or $\ell_\eta\leq 0$ for all $\eta\in R$, and $\alpha+\beta\in\langle -R\rangle$.
\end{enumerate}
This concludes the proof.
\end{proof}

\begin{prop}\label{prop:par.set}
Let $P$ be a parabolic set of $\dot{\D}$. Then $P=\dot{\D}^+(\dot{\S})\cup \langle -R\rangle$ for some set of simple roots $\dot{\S}$ and a non-empty subset $R \subseteq \dot{\S}$.
\end{prop}
\begin{proof}
Recall that we are assuming that parabolic sets are such that $P\cap(-P)=P_0\neq \emptyset$, and let $\dot{\D}^+(\dot{\S})$ be a set of positive roots contained in $P$, where $\dot{\S}$ is a set of simple roots (see Lemma~\ref{lem:pos.system.in.par.sets}). Set
	\[
R=\{-P\}\cap \{\dot{\S}\}.
	\]	
We claim $P=\dot{\D}^+(\dot{\S})\cup \langle -R\rangle$. Indeed, first of all notice that $-R\subseteq P$, and hence $\langle -R\rangle\subseteq P$, since $P$ is a additively closed. Let now $\beta\in P\setminus \dot{\D}^+(\dot{\S})$. Then $\beta=-\gamma$ for some $\gamma\in\dot{\D}^+(\dot{\S})$. We define the height of $\beta$ to be the height of $\gamma$, that is, if 
	\[
	\gamma=\sum_{\alpha\in\dot{\S}} k_\alpha\alpha, 
	\]
then the height of $\beta$ is given by $\hgt(\beta)=\hgt (\gamma)=\sum k_\alpha$. We will prove, by induction on the height of $\beta$, that $\beta\in\langle -R\rangle$. If $\hgt(\beta)=1$, then $-\gamma\in-\dot{\S}$. In particular, $-\gamma\in P\cap -\dot{\S}\subseteq \langle -R\rangle$. Suppose now $\hgt(\beta)=n$, and consider $\gamma'\in \dot{\D}^+$ with $\hgt(\gamma')=n-1$, such that $-\gamma=-\alpha-\gamma'$, for some $\alpha\in\dot{\S}$ (the existence of such $\gamma'$ can be verified case by case, see for instance \cite{FSS00} for details). Since $-\gamma'=-\gamma+\alpha\in \dot{\D}$, where $-\gamma, \alpha\in P$, we have that $-\gamma'\in P$. Similarly, $-\alpha=-\gamma+\gamma'\in \dot{\D}$ and hence $-\alpha\in P$. Since $\hgt(\alpha)=1$, we have that $-\alpha\in \langle -R\rangle$. On the other hand, by induction, we also have $-\gamma'\in\langle -R\rangle$. Then $-\gamma\in\langle -R\rangle$, since $-\alpha, -\gamma' \in\langle -R\rangle$. Hence our claim is proved.
\end{proof}

Let $\dot{\S}$ be a set of simple roots of $\dot{\D}$, and $R\subseteq \dot{\S}$ be a non-empty set. Consider the parabolic set $P=\dot{\D}^+(\dot{\S})\cup \langle -R\rangle$ (see Lemma~\ref{lem:par.set.construction}). For $\alpha\in\dot{\S}$, we define the reflection of $P$ along $\alpha$ to be
	\[
P_\alpha=\dot{\D}^+(\dot{\S}_\alpha)\cup\langle -R_\alpha\rangle,
	\]
where $R_\alpha=\{r_\alpha(\beta)\mid \beta\in R\}$. By Lemma~\ref{lem:par.set.construction}, $P_\alpha$ is also a parabolic set of $\dot{\D}$.

\begin{cor}
Let $\dot{\Pi}$ be a fixed set of simple roots of $\dot{\D}$ (for instance, a distinguished one). Then every parabolic set $P$ can be obtained by applying a chain of odd and even reflections on a parabolic set $P_{\dot{\Pi}}=\dot{\D}^+(\dot{\Pi})\cup\langle -S\rangle$, where $S$ is a suitable subset of $\dot{\Pi}$.
\end{cor}
\begin{proof}
By Proposition~\ref{prop:par.set}, every parabolic set $P$ is of the form $\dot{\D}^+(\dot{\S})\cup \langle -R\rangle$, where $\dot{\S}$ is a set of simple roots of $\dot{\D}$, and $R\subseteq \dot{\S}$ is a suitable subset. On the other hand, one can obtain $\dot{\S}$ from $\dot{\Pi}$ by applying a chain of even and odd reflections on $\dot{\Pi}$ (see \cite[Corollary~4.5]{Ser11}). Let $\varphi$ be such a chain that sends $\dot{\Pi}$ to $\dot{\S}$, and consider $S=\varphi^{-1}(R)$. Now it is clear that $\varphi$ sends $\cP_{\dot{\Pi}}$ to $P$.
\end{proof}

\begin{cor}\label{cor:par.sets.containing.D_0}
If $\g$ is of type I, then a parabolic set $P=\dot{\D}^+(\dot{\S})\cup \langle -R\rangle$ contains $\dot{\D}_{\bar 0}$ if and only if $\dot{\S}$ is a distinguished system of simple roots and $R=\dot{\S}_{\bar 0}$. If $\g$ is of type II, then there exist no parabolic set containing $\dot{\D}_{\bar 0}$.
\end{cor}
\begin{proof}
It is easy to see that the only possibility for $P$ to contain $\dot{\D}_{\bar 0}$ is if $\dot{\S}_{\g_{\bar 0}}\subseteq R$, where $\dot{\S}_{\g_{\bar 0}}$ denotes the set of simple roots of $\g_{\bar 0}$ associated to the set of positive even roots $\dot{\D}^+(\dot{\S})_{\bar 0}$. But  $|\dot{\S}_{\bar 0}|\leq |\dot{\S}_{\g_{\bar 0}}|$ for any choice of $\dot{\S}$ (see \cite[Proposition~1.6~(d)]{kac78}), and the equality holds if and only if $\g$ is of type I and $\dot{\S}$ is distinguished.
\end{proof}

\begin{rem}
If $\g$ is $\fgl(m,n)$, $\fsl(m,n)$ with $m\neq n$, or $\fosp(2,2n-2)$ for $n\geq 2$, then it follows from Corollary~\ref{cor:par.sets.containing.D_0} that the only parabolic subalgebras of $\g$ that contains $\g_{\bar 0}$ are induced from parabolic sets of the form $P_{\rm dis}:=\dot{\D}^+(\dot{\S}_{\rm dis})\cup\langle -\dot{\S}_{\bar 0}\rangle$, where $\dot{\S}_{\rm dis}$ is some distinguished system of simple roots. For the other types of Lie superalgebras, there is no such a parabolic subsuperalgebra.
\end{rem}

\subsection{The affine case}\label{sec:par.sets.affine-case}
 In this section we consider parabolic sets  and partitions for a system of roots $\Delta$ of a non-twisted affine Lie superalgebra $\cG$. In \cite[Theorem~4.13]{DFG09}, the authors described such sets from a geometric point of view. They proved that parabolic partitions and parabolic sets are described using the notion of strongly parabolic and principal parabolic sets. These definitions depend on how the parabolic sets interact with certain hyperplanes in the vector space spanned by $\D$. Our approach towards to this description is more combinatorial. We construct parabolic partitions explicitly in terms of pairs $(\dot{\S}, X)$, where $\dot{\S}$ is a system of simple roots of $\dot{\D}$, and $X\subseteq \dot{\S}$. For the description of parabolic sets we will need one more parameter that will be given by a subset $S\subseteq X$.

As usual we distinguish sets associated to $\cG$ by those associated to $\g$ by a dot above the respective set (for instance, if $\D$ denote the system of roots of $\cG$, then $\dot{\D}$ will denote the system of roots of $\g$).
Recall that the root system of $\cG$ is given by
	\[
\Delta=\{\alpha+n\delta\mid \alpha\in\dot{\Delta}\cup\{0\},\ n\in\Z\}\setminus\{0\}.
\]

We give now the two extreme examples of parabolic partitions of $\D$.

\begin{example}
Let $\S=\{\alpha_0,\alpha_1,\ldots, \alpha_n\}$ be a fixed set of simple roots of $\cG$. Such a system can be chosen such that $\S=\dot{\S}\cup \{\alpha_0=\delta-\theta\}$, where $\dot{\S}$ is a fixed system of simple roots  of $\g$, and $\theta$ is the highest root of $\g$ with respect to $\dot{\S}$. It is clear that $P_{\st}(\dot{\S})=\{\alpha+ n\delta \mid \alpha\in\dot{\Delta}\cup \{0\},\ n>0\}\cup \dot{\Delta}^+(\dot{\S})$ is a parabolic partition of $\Delta$. The Borel subalgebra induced by $P_{\st}(\dot{\S})$ is 
	\[
\cB_\st(\dot{\S})=\cH\oplus (\fn^+\otimes 1)\oplus (\g\otimes t\C[t]).
	\]
\end{example}
The Borel subalgebra in the next example cannot be obtained via a system of simple roots of $\cG$ as that of the previous example.
\begin{example}
Let $\fb(\dot{\S})=\fh\oplus \fn^+(\dot{\S})$ be a Borel subalgebra of $\g$ induced by a system of simple roots $\dot{\S}$ of $\dot{\D}$. This Borel subalgebra induces a Borel subalgebra of $\cG$ given by
	\[
\cB_\nat(\dot{\S})=\cH\oplus (\fh\otimes t\C[t])\oplus (\fn^+(\dot{\S})\otimes \C[t,t^{-1}]).
	\]
The \emph{natural} parabolic partition of $\cG$ is defined by $P_\nat(\dot{\S})=\{\alpha\in\Delta\mid \g_\alpha\subseteq \cB_\nat(\dot{\S})\}$. Explicitly, we have that $P_\nat(\dot{\S})=\{\alpha+n\delta,\ k\delta\mid \alpha\in\dot{\Delta}^+(\dot{\S}), n\in\Z,\ k\in \Z_+\}$. 
\end{example}
The partitions $P_\st(\dot{\S})$ and $P_\nat(\dot{\S})$ are not conjugate under the action of the Weyl group $W$ (recall that $W$ is by definition the Weyl group of $\g_{\bar 0}$). This follows from the fact that the existence of $w\in W$ such that $w(P_{st}(\dot{\S}))=P_{\nat}(\dot{\S})$, would provide an isomorphism between the corresponding Borel  subalgebras. However,
\begin{align*}
\cB_{\st}(\dot{\S}) & = (\fh+\C K)\oplus [\cB_{\st}(\dot{\S}), \cB_{\st}(\dot{\S})]\text{ and } \\
\cB_{\nat}(\dot{\S}) & = (\C K+\fh\otimes \C[t])\oplus  [\cB_{\nat}(\dot{\S}), \cB_{\nat}(\dot{\S})]
\end{align*}
and the existence of an isomorphism implies that $\fh+\C K\cong \C K+\fh\otimes \C[t]$, which is clearly a contradiction.

In order to give a precise description of parabolic partitions of $\D$, we need the notion of a reflection along a root of a parabolic partition $P$ of $\D$. As in the finite case, a root $\alpha\in P$ defines a reflection of $P$ if the set
	\[
r_\alpha(P) := P_\alpha = (P\setminus\{\alpha, 2\alpha\})\cup\{-\alpha, -2\alpha\}
	\]
is again a parabolic partition of $\D$. Two parabolic partitions $P$ and $P'$ are said to be adjacent if $P'=r_\alpha(P)$, for a unique root $\alpha\in P$. If $P$ and $P'$ are adjacent, then we say that $P'$ is obtained from $P$ by a reflection along $\alpha$. Following \cite{JK89}, a root $\alpha\in P$ is said to be \emph{good} if $|(\alpha+\Z_{<0}\delta)\cap P|<\infty$. A root that is not good is called \emph{bad}. Let $G(P)$ and $B(P)$ denote the sets of good and bad roots of $P$, respectively. In particular, $P=G(P)\cup B(P)$.

\begin{rem}
\begin{enumerate}
\item If $\delta\in P$, then $\delta$ is a good root of $P$.
\item Every root of $P_\st(\dot{\S})$ is good, for any choice of $\dot{\S}$.
\item The natural partition is the partition (containing $\delta$) with smallest possible set of good roots. In fact, the only good roots of $P_\nat(\dot{\S})$ are $k\delta$ with $k\in \Z_+$, and it is clear that any partition containing $\delta$ also has these roots as good roots.
\end{enumerate}
\end{rem}

Let $\dot{\S}$ be a set simple roots of $\dot{\D}$, $X\subseteq \dot{\S}$, $\dot{\D}^+(X)=\langle X\rangle$ and $\dot{\D}(X)=\langle\pm X\rangle$. We define
\begin{gather*}
P(\dot{\S},X)=\{\alpha+n\delta\mid \alpha\in \dot{\D}^+(\dot{\S})\setminus \dot{\D}^+(X),\ n\in \Z\}\ \cup \\
\{\alpha+n\delta\mid \alpha\in \dot{\D}(X)\cup \{0\},\ n\in \Z_{>0}\}\cup \dot{\D}^+(X).
\end{gather*}
In particular, we have that $P(\dot{\S},\emptyset) =P_\nat(\dot{\S})$ and $P(\dot{\S},\dot{\S})=P_\st(\dot{\S})$. Also, notice that if $X\neq \dot{\S}$, then $P(\dot{\S},X)\neq P_\st(\dot{\S}')$ for any set of simple roots $\dot{\S}'$ of $\D$. This follows from the fact that all roots of $P_\st(\dot{\S}')$ are good, for any choice of $\dot{\S}'$.
\begin{lem}\label{lem:gbroots}
Let $P$ be a parabolic partition such that $\delta\in P$. Then we have the following properties:
\begin{enumerate}
\item \label{item.lem:gbroots.1} If $\alpha,\beta,\alpha+\beta\in P$, where either $\alpha$ or $\beta$ is bad, then $\alpha+\beta$ is bad.
\item \label{item.lem:gbroots.2} If $\alpha,\beta,\alpha+\beta\in P$, where both $\alpha$ and $\beta$ are good, then $\alpha+\beta$ is good.
\item \label{item.lem:gbroots.3} If $\alpha,\beta,\alpha-\beta\in P$, where both $\alpha$ and $\beta$ are good, then $\alpha-\beta$ is good.
\item \label{item.lem:gbroots.4} If $\beta=\alpha+k \delta \in P$ is a bad root for some $k\in \Z$, then $\alpha\in P$ bad.
\item \label{item.lem:gbroots.5} Let $\beta=\alpha+k\delta\in P$ with $\alpha\in \dot{\D}$ and $k\in \Z$. If $\beta$ is good, then either $\alpha$ or $-\alpha$ is a good root of $P$.
\end{enumerate}
\end{lem}
\begin{proof}
The proof of \eqref{item.lem:gbroots.1}-\eqref{item.lem:gbroots.4} can be found in \cite[Lemma~2.2]{JK89}. Part \eqref{item.lem:gbroots.5}: if $\alpha\in P$ is bad, then $\alpha+n\delta\in P$ for all $n\in \Z$, which implies that $\beta+n\delta\in P$ for all $n\in \Z$. Hence $\beta$ is bad. If $-\alpha\in P$ is bad, then, in particular $-\alpha-(n_\beta+1)\delta\in P$, which implies that $(-\alpha-(n_\beta+1)\delta)+\beta=-\delta\in P$ what is a contradiction, since $\delta\in P$. Thus $-\alpha$ must be good. 
\end{proof}

\begin{lem}\label{lem:groots}
Let $P$ be a parabolic partition of $\D$, $\dot{\S}$ be a system of simple roots of $\dot{\D}$ such that $\dot{\D}^+(\dot{\S})\subseteq P$, and $P_\st : = P_\st(\dot{\S})$. Then $|G(P)\cap -P_\st|<\infty$. Moreover, if $|G(P)\cap -P_\st|>0$, then there exists a good root $\beta=\alpha+s\delta$ such that $r_\beta(P):=P'$ is well defined and $|G(P')\cap -P_\st|< |G(P)\cap -P_\st|$.
\end{lem}
\begin{proof}
Let 
	\[
X(P)=\{\alpha\in \dot{\S}\mid \alpha \text{ is a good root of }P\},\text{ and }\D^+(X(P))=\langle X(P)\rangle.
	\] 
Notice that for each $\alpha\in \D^+(X(P))$ there exists an integer $s_\alpha\in\Z_{\leq 0}$ such that $(\alpha+\Z\delta)\cap P=\{\alpha+k\delta\mid k\geq s_\alpha\}$. It follows now from Lemma~\ref{lem:gbroots} along with the fact that $\dot{\D}^+(\dot{\S})\subseteq P$, that the set of good roots of $P$ is as follows
	\[
G(P)=\{\alpha+n_\alpha\delta,\ -\alpha+m_\alpha\delta,\ k\delta\mid \alpha\in \D^+(X(P)),\ n_\alpha\geq s_\alpha,\ m_\alpha > -s_\alpha,\ k>0\}.
	\]
In particular, 
	\[
|G(P)\cap -P_\st|=\sum_{\alpha\in \D^+(X(P))}|s_\alpha|<\infty.
	\] 
	
If $X(P)=\emptyset$, then $P=P(\dot{\S},\emptyset)=P_\nat(\dot{\S})$ and clearly $G(P) \subseteq P_\st$. Suppose that  $X(P)\neq \emptyset$, and define
\begin{align*}
\D^+(X(P))_z & = \{\alpha\in \D^+(X(P))\mid \hgt(\alpha)=z\},\quad \text{for all}\quad z\in \Z, \\ 
D_z(P)& = \{\alpha\in \D^+(X(P))_z\mid \alpha=\beta+\gamma, \\
& \beta\in \D^+(X(P))_{z-1},\ \gamma\in X(P), \text{ and }s_\alpha=s_{\beta}+s_{\gamma}\}\quad  \text{for all}\quad z\in \Z_{\geq 2}, \\
m_P& = \max\{z\in \Z_+\mid \D^+(X(P))_z\setminus D_z(P)\neq \emptyset\}.
\end{align*}

If $\D^+(X(P))_z\setminus D_z(P)= \emptyset$, for all $z\geq 2$, then the reflection $r_{\alpha-s_\alpha\delta}$ is well defined for any $\alpha\in X(P)$. On the other hand, if $\D^+(X(P))_z\setminus D_z(P)\neq \emptyset$, for some $z\in \Z_{\geq 2}$, then the reflection $r_{\alpha+s_\alpha\delta}$ is well defined for any $\alpha\in \D^+(X(P))_{m_P}\setminus D_{m_P}(P)$. In any case, we get a parabolic partition $P'$ for which $|G(P')\cap -P_\st|<|G(P)\cap -P_\st|$.
\end{proof}
\begin{cor}\label{cor:groots}
Let $P$ be a parabolic partition of $\D$, and $\dot{\S}$ be a system of simple roots of $\dot{\D}$ such that $\dot{\D}^+(\dot{\S})\subseteq P$. Then there exists a chain of even and odd reflections that sends $P$ to a parabolic partition $P'$ such that $G(P')\subseteq P_\st(\dot{\S})$.
\end{cor}
\begin{example}
Let $\cG = \frak{osp}(2,4)^{\widehat{ \ }}$, and let $\dot{\S} = \{\alpha_1, \alpha_2, \alpha_3\}$ be a distinguished set of simple roots of $\frak{osp}(2,4)$. If one denotes by $\theta$ the highest root of $\g$ with respect to $\dot{\S}$, then the distinguished Dynkin diagrams are given as follows
	\[ 
\frak{osp}(2,4) :   \arraycolsep=-1pt \qquad \begin{array}[]{rcccccccccccccccccccccc}  
    \bigotimes & - & \bigcirc &  \Leftarrow & \bigcirc \\
   \alpha_1 && \alpha_2 && \alpha_3  \end{array} \qquad \qquad \displaystyle{\theta = \alpha_1  +2 \alpha_2 + \alpha_3}
    \]
	\[ 
\frak{osp}(2,4)^{\widehat{ \ }} :   \arraycolsep=-1pt \qquad \begin{array}[]{rcccccccccccccccccccccc}

\alpha_0 &  \\
  \bigotimes\\[-.5ex]
| & \diagdown \\
      \bigotimes & - & \bigcirc & \Leftarrow & \bigcirc \\
   \alpha_1 && \alpha_2 && \alpha_3 \end{array} \qquad \qquad \displaystyle{\delta = \alpha_0 + \alpha_1 +2 \alpha_2 + \alpha_3}.
    \]
The distinguished set of positive roots $\dot{\D}^+(\dot{\S})$ of $\frak{osp}(2,4)$ is
\begin{gather*}
\dot{\D}^+(\dot{\S})=\{\alpha_1,\ \alpha_2,\ \alpha_3,\ \alpha_{1}+\alpha_2,\ \alpha_2+\alpha_3,\ \alpha_1+\alpha_2+\alpha_3,\ 2\alpha_2+\alpha_3,\ \theta\}, \\
\dot{\D}^+_{\bar 0}(\dot{\S})=\{2\alpha_2+\alpha_3,\ \alpha_3,\ \alpha_2+\alpha_3,\ \alpha_2\}, \quad \dot{\D}^+_{\bar 1}(\dot{\S})=\{\alpha_1,\ \alpha_1+\alpha_2,\ \alpha_2+\alpha_3,\ \theta\}.
\end{gather*}
The set 
\begin{align*}
P & = \{\alpha_1 + n\delta\mid n\in \Z_{\geq 0}\}\cup\{-\alpha_1+n\delta\mid n\in\Z_{\geq 1}\} \\
& \cup \{\alpha_2 + n\delta\mid n\in \Z_{\geq 0}\}\cup\{-\alpha_2+n\delta\mid n\in\Z_{\geq 1}\} \\
& \cup \{\alpha_3 + n\delta\mid n\in \Z_{\geq -1}\}\cup\{-\alpha_3+n\delta\mid n\in\Z_{\geq 2}\} \\
& \cup \{\alpha_1+\alpha_2 + n\delta\mid n\in \Z_{\geq 0}\}\cup \{-\alpha_1-\alpha_2 + n\delta\mid n\in \Z_{\geq 1}\} \\
& \cup \{\alpha_2+\alpha_3 + n\delta\mid n\in \Z_{\geq -1}\}\cup \{-\alpha_2-\alpha_3 + n\delta\mid n\in \Z_{\geq 2}\} \\
& \cup \{\alpha_1+\alpha_2+\alpha_3 + n\delta\mid n\in \Z_{\geq -2}\}\cup \{-\alpha_1-\alpha_2 - \alpha_3 + n\delta\mid n\in \Z_{\geq 3}\} \\
& \cup \{2\alpha_2+\alpha_3 + n\delta\mid n\in \Z_{\geq -1}\}\cup \{-2\alpha_2-\alpha_3 + n\delta\mid n\in \Z_{\geq 2}\} \\
& \cup \{\theta + n\delta\mid n\in \Z_{\geq -2}\}\cup \{-\theta+ n\delta\mid n\in \Z_{\geq 3}\} \\
\end{align*}
is a parabolic partition of $\D$, where $G(P)=P$ and $X=\dot{\S}$. A chain of reflections satisfying Corollary~\ref{cor:groots} can be given by
\begin{gather*}
P \overset{r_{\alpha_1+\alpha_2+\alpha_3-2\delta}}{\rightarrow} P_1 \overset{r_{\theta-2\delta}}{\rightarrow} P_2 \overset{r_{\alpha_3-\delta}}{\rightarrow} P_3 \\ 
\overset{r_{\alpha_2+\alpha_3-\delta}}{\rightarrow} P_4  \overset{r_{2\alpha_2+\alpha_3-\delta}}{\rightarrow} P_5 \overset{r_{\alpha_1+\alpha_2+\alpha_3-\delta}}{\rightarrow} P_6 \overset{r_{\theta-\delta}}{\rightarrow} P_7=P'.
\end{gather*}
Observe that such a chain is not unique. Indeed, 
\begin{gather*}
P \overset{r_{\alpha_1+\alpha_2+\alpha_3-2\delta}}{\rightarrow} P_1 \overset{r_{\theta-2\delta}}{\rightarrow} P_2 \overset{r_{\alpha_3-\delta}}{\rightarrow} P_3 \\ 
\overset{r_{\alpha_2+\alpha_3-\delta}}{\rightarrow} P_4  \overset{r_{\alpha_1+\alpha_2+\alpha_3-\delta}}{\rightarrow} P_5 \overset{r_{2\alpha_2+\alpha_3-\delta}}{\rightarrow} P_6 \overset{r_{\theta-\delta}}{\rightarrow} P_7=P'.
\end{gather*}
is another chain satisfying Corollary~\ref{cor:groots}.

\details{
\begin{align*}
P \overset{r_{\alpha_1+\alpha_2+\alpha_3-2\delta}}{\rightarrow} P_1 & = \{\alpha_1,\ldots\}\cup\{-\alpha_1+\delta,\ldots\} \\
& \cup \{\alpha_2,\ldots\}\cup\{-\alpha_2+\delta,\ldots\} \\
& \cup \{\alpha_3-\delta,\ldots\}\cup\{-\alpha_3+2\delta,\ldots\} \\
& \cup \{\alpha_1+\alpha_2,\ldots\}\cup\{-\alpha_1-\alpha_2+\delta,\ldots\} \\
& \cup \{\alpha_2+\alpha_3-\delta,\ldots\}\cup\{-\alpha_2-\alpha_3+2\delta,\ldots\} \\
& \cup \{\alpha_1+\alpha_2+\alpha_3-\delta,\ldots\}\cup\{-\alpha_1-\alpha_2-\alpha_3+2\delta,\ldots\} \\
& \cup \{2\alpha_2+\alpha_3-\delta,\ldots\}\cup\{-2\alpha_2-\alpha_3+2\delta,\ldots\} \\
& \cup \{\theta-2\delta,\ldots\}\cup\{-\theta+3\delta,\ldots\} \\
\end{align*}
\begin{align*}
P_1 \overset{r_{\theta-2\delta}}{\rightarrow} P_2 & = \{\alpha_1,\ldots\}\cup\{-\alpha_1+\delta,\ldots\} \\
& \cup \{\alpha_2,\ldots\}\cup\{-\alpha_2+\delta,\ldots\} \\
& \cup \{\alpha_3-\delta,\ldots\}\cup\{-\alpha_3+2\delta,\ldots\} \\
& \cup \{\alpha_1+\alpha_2,\ldots\}\cup\{-\alpha_1-\alpha_2+\delta,\ldots\} \\
& \cup \{\alpha_2+\alpha_3-\delta,\ldots\}\cup\{-\alpha_2-\alpha_3+2\delta,\ldots\} \\
& \cup \{\alpha_1+\alpha_2+\alpha_3-\delta,\ldots\}\cup\{-\alpha_1-\alpha_2-\alpha_3+2\delta,\ldots\} \\
& \cup \{2\alpha_2+\alpha_3-\delta,\ldots\}\cup\{-2\alpha_2-\alpha_3+2\delta,\ldots\} \\
& \cup \{\theta-\delta,\ldots\}\cup\{-\theta+2\delta,\ldots\} \\
\end{align*}
\begin{align*}
P_2 \overset{r_{\alpha_3-\delta}}{\rightarrow} P_3 & = \{\alpha_1,\ldots\}\cup\{-\alpha_1+\delta,\ldots\} \\
& \cup \{\alpha_2,\ldots\}\cup\{-\alpha_2+\delta,\ldots\} \\
& \cup \{\alpha_3,\ldots\}\cup\{-\alpha_3+\delta,\ldots\} \\
& \cup \{\alpha_1+\alpha_2,\ldots\}\cup\{-\alpha_1-\alpha_2+\delta,\ldots\} \\
& \cup \{\alpha_2+\alpha_3-\delta,\ldots\}\cup\{-\alpha_2-\alpha_3+2\delta,\ldots\} \\
& \cup \{\alpha_1+\alpha_2+\alpha_3-\delta,\ldots\}\cup\{-\alpha_1-\alpha_2-\alpha_3+2\delta,\ldots\} \\
& \cup \{2\alpha_2+\alpha_3-\delta,\ldots\}\cup\{-2\alpha_2-\alpha_3+2\delta,\ldots\} \\
& \cup \{\theta-\delta,\ldots\}\cup\{-\theta+2\delta,\ldots\} \\
\end{align*}
\begin{align*}
P_3 \overset{r_{\alpha_2+\alpha_3-\delta}}{\rightarrow} P_4 & = \{\alpha_1,\ldots\}\cup\{-\alpha_1+\delta,\ldots\} \\
& \cup \{\alpha_2,\ldots\}\cup\{-\alpha_2+\delta,\ldots\} \\
& \cup \{\alpha_3,\ldots\}\cup\{-\alpha_3+\delta,\ldots\} \\
& \cup \{\alpha_1+\alpha_2,\ldots\}\cup\{-\alpha_1-\alpha_2+\delta,\ldots\} \\
& \cup \{\alpha_2+\alpha_3,\ldots\}\cup\{-\alpha_2-\alpha_3+\delta,\ldots\} \\
& \cup \{\alpha_1+\alpha_2+\alpha_3-\delta,\ldots\}\cup\{-\alpha_1-\alpha_2-\alpha_3+2\delta,\ldots\} \\
& \cup \{2\alpha_2+\alpha_3-\delta,\ldots\}\cup\{-2\alpha_2-\alpha_3+2\delta,\ldots\} \\
& \cup \{\theta-\delta,\ldots\}\cup\{-\theta+2\delta,\ldots\} \\
\end{align*}
\begin{align*}
P_4 \overset{r_{2\alpha_2+\alpha_3-\delta}}{\rightarrow} P_5 & = \{\alpha_1,\ldots\}\cup\{-\alpha_1+\delta,\ldots\} \\
& \cup \{\alpha_2,\ldots\}\cup\{-\alpha_2+\delta,\ldots\} \\
& \cup \{\alpha_3,\ldots\}\cup\{-\alpha_3+\delta,\ldots\} \\
& \cup \{\alpha_1+\alpha_2,\ldots\}\cup\{-\alpha_1-\alpha_2+\delta,\ldots\} \\
& \cup \{\alpha_2+\alpha_3,\ldots\}\cup\{-\alpha_2-\alpha_3+\delta,\ldots\} \\
& \cup \{\alpha_1+\alpha_2+\alpha_3-\delta,\ldots\}\cup\{-\alpha_1-\alpha_2-\alpha_3+2\delta,\ldots\} \\
& \cup \{2\alpha_2+\alpha_3,\ldots\}\cup\{-2\alpha_2-\alpha_3+\delta,\ldots\} \\
& \cup \{\theta-\delta,\ldots\}\cup\{-\theta+2\delta,\ldots\} \\
\end{align*}
\begin{align*}
P_5 \overset{r_{\alpha_1+\alpha_2+\alpha_3-\delta}}{\rightarrow} P_6 & = \{\alpha_1,\ldots\}\cup\{-\alpha_1+\delta,\ldots\} \\
& \cup \{\alpha_2,\ldots\}\cup\{-\alpha_2+\delta,\ldots\} \\
& \cup \{\alpha_3,\ldots\}\cup\{-\alpha_3+\delta,\ldots\} \\
& \cup \{\alpha_1+\alpha_2,\ldots\}\cup\{-\alpha_1-\alpha_2+\delta,\ldots\} \\
& \cup \{\alpha_2+\alpha_3,\ldots\}\cup\{-\alpha_2-\alpha_3+\delta,\ldots\} \\
& \cup \{\alpha_1+\alpha_2+\alpha_3,\ldots\}\cup\{-\alpha_1-\alpha_2-\alpha_3+\delta,\ldots\} \\
& \cup \{2\alpha_2+\alpha_3,\ldots\}\cup\{-2\alpha_2-\alpha_3+\delta,\ldots\} \\
& \cup \{\theta-\delta,\ldots\}\cup\{-\theta+2\delta,\ldots\} \\
\end{align*}
\begin{align*}
P_6 \overset{r_{\theta-\delta}}{\rightarrow} P_7 & = \{\alpha_1,\ldots\}\cup\{-\alpha_1+\delta,\ldots\} \\
& \cup \{\alpha_2,\ldots\}\cup\{-\alpha_2+\delta,\ldots\} \\
& \cup \{\alpha_3,\ldots\}\cup\{-\alpha_3+\delta,\ldots\} \\
& \cup \{\alpha_1+\alpha_2,\ldots\}\cup\{-\alpha_1-\alpha_2+\delta,\ldots\} \\
& \cup \{\alpha_2+\alpha_3,\ldots\}\cup\{-\alpha_2-\alpha_3+\delta,\ldots\} \\
& \cup \{\alpha_1+\alpha_2+\alpha_3,\ldots\}\cup\{-\alpha_1-\alpha_2-\alpha_3+\delta,\ldots\} \\
& \cup \{2\alpha_2+\alpha_3,\ldots\}\cup\{-2\alpha_2-\alpha_3+\delta,\ldots\} \\
& \cup \{\theta,\ldots\}\cup\{-\theta+\delta,\ldots\} \\
\end{align*}
}
\end{example}
The next result is a generalization of \cite[Proposition~2.1]{JK89} and \cite[Theorem~2.6]{Fut92} to the super setting.
\begin{theo}\label{prop:par.part.}
Let $P$ be a parabolic partition of $\D$. Then, up to a chain of odd and even reflections, we have that $P=P(\dot{\S},X)$, for some system of simple roots $\dot{\S}$ of $\dot{\D}$ and a subset $X\subseteq \dot{\S}$.
\end{theo}
\begin{proof}
Notice that $\D'=P\cap \dot{\D}$ is a parabolic partition of $\dot{\D}$. Then there exists a system of simple roots $\dot{\S}\subseteq \D'$ such that $\D'=\dot{\D}^+(\dot{\S})$. In particular, $\dot{\S}\subseteq P$. Let $\alpha_0\in\Delta$ such that $\S=\dot{\S}\cup\{\alpha_0\}$ is a system of simple roots of $\Delta$ and consider $P_\st=P_\st(\dot{\S})$. By Corollary~\ref{cor:groots}, exists a chain of reflections that sends $P$ to a parabolic partition $P'$ such that $G(P')\subseteq P_\st$. In other words, up to applying a chain of reflections, we may assume that $G(P)\subseteq P_\st$. Let $X(P)$ be as in Lemma~\ref{lem:groots} and set $\D(X(P))=\langle\pm X(P)\rangle$. Under the assumption that $G(P)\subseteq P_\st$, one can write the set of good roots of $P$ as follows
	\[
G(P)=\{\alpha+n\delta\mid \alpha\in  \D(X(P))\cup\{0\},\  n\in \Z_{ > 0}\}\cup \D^+(X(P)).
	\]
This implies that $B(P)=\{\alpha+n\delta\mid \alpha\in\dot{\D}^+(\dot{\S})\setminus \dot{\D}^+(X),\ n\in \Z\}$ is the set of bad roots of $P$, and hence
	\[
P  =G(P)\cup B(P)=P(\dot{\S},X(P))
	\]
as desired.
\end{proof}

Recall that throughout the paper we reserve the term parabolic set (resp. parabolic partition) for the case  $P_0=P\cap -P \neq \emptyset$ (resp. $P_0=P\cap -P = \emptyset$). The following result describes all parabolic sets of $\D$.

\begin{prop}\label{prop:delta.notin.P_0}
Let $P$ be a parabolic set of $\D$ such that $\delta\in P\setminus (-P)$. Then there exists a system of simple roots $\dot{\S}$ of $\dot{\D}$, nonempty subsets $X\subseteq \dot{\S}$ and $S\subseteq X$, and a chain of even and odd reflections $\varphi$, such that $P=P_\varphi (\dot{\S}, X)\cup \langle -S\rangle$.
\end{prop}

\begin{proof}
Let $P'$ be a parabolic partition contained in $P$. By Proposition~\ref{prop:par.part.}, there exists a chain of even and odd reflections $\varphi$ such that $P' = P_\varphi (\dot{\S},X)$. Consider $\beta \in P\setminus P_\varphi(\dot{\S},X)$. Then $\beta\in -P_\varphi(\dot{\S},X)$ and we have the following possibilities: 
\begin{enumerate}
\item \label{item1:prop:delta.notin.P_0} $\beta=-\alpha+n_\beta\delta$ for some $\alpha\in\dot{\D}^+(\dot{\S})\setminus \dot{\D}^+(X)$ and some $n_\beta\in \Z$. Since $\alpha\in\dot{\D}^+(\dot{\S})\setminus \dot{\D}^+(X)$, we have that $\alpha+n\delta\in P_\varphi(\dot{\S},X)\subseteq P$ for all $n\in \Z$. In particular, $(-\alpha+n_\beta\delta)+(\alpha+(-n_\beta-1)\delta)=-\delta\in P$, which is a contradiction.
\item \label{item2:prop:delta.notin.P_0} $\beta=-\alpha-n_\beta\delta$ for some $\alpha\in\dot{\D}(X)\cup\{0\}$ and some $n_\beta\in \Z_{\geq k}$ with $k\in \Z_{\geq 0}$. If $\alpha\in \dot{\D}(X)$, then we have that $\alpha+n\delta\in P_\varphi(\dot{\S},X)\subseteq P$ for all $n\in \Z_{\geq -(k-1)}$, and hence $(-\alpha-n_\beta\delta)+(\alpha+(n_\beta-1)\delta)=-\delta\in P$, which is a contradiction. If $\alpha=0$, then $\beta=-n_\beta\delta\in P$, with $n_\beta\in \Z_{>0}$, and hence $(-n_\beta)\delta+(n_\beta-1)\delta=-\delta\in P$, which is again a contradiction.
\end{enumerate}
By \eqref{item1:prop:delta.notin.P_0}~and~\eqref{item2:prop:delta.notin.P_0}, we conclude that the only possibility is that $\beta=-\alpha$, for some $\alpha\in \dot{\D}^+(X)$. Since such a $\beta$ exists ($P\cap -P\neq \emptyset$), we must have $X\neq \emptyset$ and $\{-P\}\cap X=:S\neq \emptyset$. Now it is clear that $P\setminus P_\varphi (\dot{\S},X)=\langle -S\rangle$, and the result follows.
\end{proof}

\begin{prop}\label{prop:delta.in.P_0} 
Let $P$ be a parabolic set of $\D$ such that $\delta\in P\cap -P$. Then there exists a system of simple roots $\dot{\S}$ of $\dot{\D}$ and $S\subseteq \dot{\S}$ such that
	\[
P= P(\dot{\S},\emptyset)\cup \langle \{-S\} \cup\{\pm \delta\}\rangle.
	\]
\end{prop}

\begin{proof}
Let $\dot{\D}^+(\dot{\S})$ be a parabolic partition of $\dot{\D}$ contained in $P$. Then $\alpha+n\delta\in P$ for all $\alpha\in \dot{\D}^+(\dot{\S})$ and $n\in \Z$ (since $\pm\delta\in P$), which implies $P(\dot{\S},\emptyset)\subseteq P$. Now, notice that if $\beta=\alpha+n\delta\in P\setminus P(\dot{\S}, \emptyset)$, then $\beta\in -P(\dot{\S}, \emptyset)$ and hence $\beta=-\alpha+n\delta$ for some $\alpha\in \dot{\D}^+(\dot{\S})\subseteq P$ and some $n\in \Z$. Since $\pm\delta\in P$, we have that $-\alpha\in P$. Thus $\pm \alpha\in P$. Let $-S=P\cap (-\dot{\S})$. We claim that $-\alpha\in \langle -S \rangle$, for every $\alpha\in \dot{\D}^+(\dot{\S})$ such that $\pm\alpha\in P$. Indeed, if $\hgt(\alpha)=1$, then $-\alpha\in P\cap\{-\dot{\S}\}=-S$.  If $\alpha$ is not simple, then there exists a simple root $\gamma\in\dot{\S}$ such that $\alpha-\gamma\in\dot{\D}^+(\dot{\S})\subseteq P$. Moreover $-(\alpha-\gamma)=-\alpha+\gamma \in P$. Then $\pm(\alpha-\gamma)\in P$, and $\hgt(\alpha-\gamma)<\hgt(\alpha)$. Thus, by induction, $-(\alpha-\gamma)\in \langle -S\rangle$. On the other hand, $-\gamma=(\alpha-\gamma)-\alpha\in P\cap (-\dot{\S})=-S$. Therefore $-\alpha=-(\alpha-\gamma)-\gamma \in \langle -S\rangle$, and the claim is proved. This shows that $P\subseteq P(\dot{\S},\emptyset)\cup \langle \{-S\}\cup\{\pm \delta\}\rangle$. Since the other inclusion is clear, the result is proved.
\end{proof}

\section{Parabolic and Borel subalgebras}

Recall that throughout the paper we reserve the term parabolic set (resp. parabolic partition) for the case  $P_0=P\cap -P \neq \emptyset$ (resp. $P_0=P\cap -P = \emptyset$).

If $P$ is a parabolic partition, then
	\[
\cN_P^\pm=\bigoplus_{\alpha\in \pm P}\cG_\alpha \quad\text{and}\quad \cB_P^\pm=\cH \oplus \cN_P^\pm
	\]
are subalgebras of $\cG$.  If $P$ is a parabolic set, then
	\[
\cG^0=\bigoplus_{\alpha\in P_0}\cG_\alpha,\quad \cU_P^\pm=\bigoplus_{\alpha\in \pm P\setminus P_0}\cG_\alpha\quad\text{and}\quad \cP_P^\pm = (\cH + \cG^0) \oplus \cU_P^\pm
	\]
are subalgebras of $\cG$.  Throughout the paper we will drop the $+$ from the notation above, that is, we write $\cN_P:=\cN_P^+$, $\cB_P:=\cB_P^+$, $\cU_P = \cU_P^+$, and $\cP_P=\cP_P^+$. The superalgebra $\cB_P$ (resp. $\cP_P$) is called a \emph{Borel} (resp. \emph{parabolic}) subalgebra of $\cG$ associated to $P$. 
In the non-super setting, if $P$ is a parabolic partition, then the Borel subalgebra $\cB_P=\cH\oplus \cN_P$ is solvable if and only if $P=P_\nat$. Due to the existence of isotropic odd roots, this is no longer the case in the super setting.
\begin{prop}
Let $X=\cup_{i=1}^m X_i$ be the decomposition of $X$ into connected components, and $P=P(\dot{\S}, X)$. The Borel subalgebra $\cB_P$ is solvable if and only if either $X=\emptyset$, or $|X_i|=1$ for $i=1,\ldots, m$, in which case the only root of $X_i$ must be isotropic.
\end{prop}
\begin{proof}
The case $X=\emptyset$ is clear. Suppose then $X\neq \emptyset$ and that some $X_i$ has at least one non-isotropic root $\beta$. Now we have two cases: either $\langle \pm \beta\rangle=\{\pm \beta\}$, and we have a copy of $\frak{sl}(2)\otimes t\C[t]$ in $\cB_P$; or $\langle  \pm \beta\rangle=\{\pm \beta, \pm 2\beta\}$, and we have a copy of $\frak{osp}(1,2)\otimes t\C[t]$ in $\cB_P$. In both cases one see that $\cB_P$ is not solvable. Suppose now that all roots in $X_i$ are isotropic. If $|X_i|>1$, then it has two isotropic roots that are connected. In this case we have a copy of $\frak{sl}(1,2)\otimes t\C[t]$ in $\cB_P$, which implies that $\cB_P$ is not solvable. Finally, if $|X_i|=1$, then the only element in $X_i$ yields a copy of $\frak{sl}(1,1)\otimes t\C[t]$, which is nilpotent. This along with the fact that $\cB_{P(\dot{\S}, \emptyset)}$ is solvable implies that $\cB_P$ is solvable (the peaces in $\cB_P$ that are not like $\frak{sl}(1,1)\otimes t\C[t]$ are like $\cB_{P(\dot{\S}, \emptyset)}$).
\end{proof}

Let $P$ be a parabolic set, and let $\cG^0$ be the Lie subalgebra generated by the root spaces $\cG_\alpha$, with $\alpha\in P_0$. Further, notice that
	\[
\fH := \C K\oplus \bigoplus_{k\in \Z \setminus \{0\}}  \cG_{k\delta}
	\]
is a Heisenberg Lie algebra contained in $\cG$. The next result describes the structure of $\cG^0$.

\begin{theo}\label{thm:descr.of.G_0}
\begin{enumerate} 
\item \label{item.thm:descr.of.G_0.1} If $\delta\notin P_0$, then
	\[
\cG^0 = \bigoplus_{i=1}^m \frak{g}^i,
	\]
where each $\frak{g}^i$ is either isomorphic to $\fsl(m,n)$ with $m\neq n$, $\fosp(m,n)$, $D(2,1;\alpha)$ ($a\neq 0,-1$), $G(3)$, $F(4)$, or to $\frak{sl}(1,1)$, and $[\frak{g}^i,\frak{g}^j]=0$ if $i\neq j$.
\item \label{item.thm:descr.of.G_0.2} If $\delta\in P_0$, then $\cG^0=\cG^1+\cG^2$, with
	\[
\cG^1=\bigoplus_{i=1}^m L(\g^i)\oplus \C K,
	\]
where each $\frak{g}^i$ is either isomorphic to $\fsl(m,n)$ with $m\neq n$, $\fosp(m,n)$, $D(2,1;\alpha)$ ($a\neq 0,-1$), $G(3)$, $F(4)$, or to $\frak{sl}(1,1)$, $[\frak{g}^i,\frak{g}^j]=0$ if $i\neq j$; and $\cG^2\subseteq \fH$ is such that $\cG^2+(\cG^1\cap \fH)=\fH$. Moreover, $\cG^2$ can be chosen so that $[\cG^1, \cG^2]=0$ if and only if $\frak{g}^i\ncong \frak{sl}(1,1)$, for all $i=1,\dots, m$.
\end{enumerate}
\end{theo}
\begin{proof}
For part \eqref{item.thm:descr.of.G_0.1}, Proposition~\ref{prop:delta.notin.P_0} implies that $P=P(\dot{\S}, X)\cup \langle -S\rangle$, where $S\in \dot{\S}$. Let $S=\cup S_i$ be the connected decomposition of $S$, that is $(S_i | S_j)=0$ if and only if $i\neq j$. Define $\frak{g}^i$ to be the subalgebra generated by $\frak{g}_\alpha$, for $\alpha\in \langle \pm S_i\rangle$. For part \eqref{item.thm:descr.of.G_0.2}, notice that, by Proposition~\ref{prop:delta.in.P_0} we can assume that $P_0=\langle\{\pm S\}\cup\{\pm \delta\}\rangle$ for some $S\subseteq \dot{\S}$. Then the superalgebras $\frak{g}^i$ are defined as in part \eqref{item.thm:descr.of.G_0.1}. Now we can complete the linear independent set $\{h_\alpha\mid \alpha\in S\}$ to a basis for $\frak{h}$. If $F$ denotes such a completion, then we define
	\[
\cG^2=\bigoplus_{\substack{k\in \Z \setminus \{0\} \\ h \in F}} \C h\otimes t^k.
	\] 

Notice that  $\frak{g}^i\cong \frak{sl}(1,1)$, is equivalent to  $S_i=\{\alpha\}$, where $\alpha$ is an isotropic root. Since $(S_i | S_j)=0$ for $i\neq j$, it follows from the non-degeneracy of $(\cdot | \cdot)$ that does not exist a set $F$, as above, such that $(S_i | F)=0$. Thus there is no $\cG^2$ such that $[\cG^1, \cG^2]=0$. When $\frak{g}^i\ncong \frak{sl}(1,1)$ for all $i=1,\ldots, m$, then for every root $\alpha$ of $S_i$, there exists a root $\beta$ of $S_i$ such that $(h_\alpha | h_\beta)\neq 0$ (note that $\beta$ can be equal $\alpha$ in various cases). In particular, $(\cdot | \cdot )|_{S}$ is non-degenerate, and hence we can choose $F$ such that $(F | S)=0$. But this implies $[\cG^1,\cG^2]=0$.
\end{proof}

As a consequence of the results above, we have (up to a chain of odd and even reflections) a description of all Borel and parabolic subalgebras of $\cG$.

\begin{theo}\label{thm:parabolic.structure}
Let $P$ be a parabolic set containing the parabolic partition $P(\dot{\S}, X)$, and let $\cP_{P}$ be the corresponding parabolic subalgebra of $\cG$.  Then $\cP_P$ admits a decomposition
	\[
\cP_P=(\cG^0+\cH)\oplus \cU_P,
	\]
where $\cU_P=\bigoplus_{\alpha\in P\setminus P_0} \cG_{\alpha}$ is an ideal in $\cP_P$, and $\cG^0$ in given as in Theorem~\ref{thm:descr.of.G_0}.
\end{theo}

\begin{rem}
Notice that, if the set $S$ in the proof of Theorem~\ref{thm:descr.of.G_0} contains some isotropic root $\alpha$, then the non-degeneracy of $(\cdot | \cdot)$ does not imply the existence of a set $F$ of $\frak{h}$ such that $(h_\alpha | F)=0$, and $F\cup \{h_\alpha | \alpha\in S\}$ is a basis of $\frak{h}$. This kind of behavior does not occur in the non-super setting. The next example illustrates this.
\end{rem}

\begin{example}\label{ex:par.sets.sl(1,2)}
Consider the Lie superalgebra $\frak{g}=\frak{sl}(1,2)$. The roots of $\g$ (using the notation of \cite{FSS00}) are $\pm \alpha = \pm (\delta - \varepsilon_1)$, $\pm \beta = \pm (\varepsilon_1 - \varepsilon_2)$ and $\pm (\alpha + \beta)$. Notice that $\alpha$ is an odd isotropic root (that is, $\alpha(h_\alpha)=(\alpha|\alpha)=0$), and that $\beta$ is an even root. Recall that, up to $W$-conjugation, the only possible choices of system of simple roots $\dot{\S}$ are: $\{\alpha, \beta\}$, $\{-\alpha,\alpha + \beta\}$, and $\{\beta, -\alpha-\beta\}$. Fixing  $\dot{\S}=\{\alpha,\beta\}$, we have that $\dot{\D}^+(\dot{\S})=\{\alpha,\beta,\alpha+\beta\}$. By Proposition~\ref{prop:par.part.}, up to a chain of odd and even reflections, the parabolic partitions of $\cG$ that contain $\dot{\D}^+(\dot{\S})$ are of the form $P(\dot{\S}, X)$, where $X= \emptyset$, $\{\alpha\}$, $\{\beta\}$ or $\dot{\S}$. The most interesting parabolic sets $P$ to consider are those that contain $P_\nat=P(\dot{\S},\emptyset)$. In this case we have two possibilities for $P_0$.
\begin{enumerate}
\item \label{case:sl(2)} If $S=\{\beta\}$, then $P_0=\{\pm \beta + n\delta,\ \pm k \delta\mid n\in \Z,\ k\in \Z_{>0}\}$. Thus
	\[
\cG^1=\frak{sl}(2)^{\widehat{ \ }} := \frak{sl}(2)\otimes \C[t,t^{-1}]\oplus \C K.
	\]
Notice now that $S$ does not contain any isotropic root, and the elements  
	\[
h_\beta =
\left(
\begin{array}{c|cc}
0 & 0 & 0 \\
\hline
0 & 1 & 0 \\
0 & 0 & -1
\end{array}
\right),\quad 
h =
\left(
\begin{array}{c|cc}
2 & 0 & 0 \\
\hline
0 & 1 & 0 \\
0 & 0 & 1
\end{array}
\right)
	\]
provide an orthogonal basis of $\frak{h}$ (the non-degenerate bilinear form here is $(X|Y)={\rm str}(XY)$, where ${\rm str}$ stands for the super-trace of a matrix). Then, defining
	\[
\cG^2=\fH_\alpha := \C K\oplus \bigoplus_{k\in \Z \setminus \{0\}} (\C h \otimes t^k),
	\]
we obtain
	\[
\cG^0=\frak{sl}(2)^{\widehat{ \ }} + \fH_\alpha,\quad \frak{sl}(2)^{\widehat{ \ }} \cap \fH_\alpha=\C K,
	\]
and $[\cG^1, \cG^2]=0$.
\item \label{case:sl(1,1)} If $S=\{\alpha\}$, then $P_0=\{\pm \alpha + n \delta,\ \pm k \delta\mid n\in \Z,\ k\in \Z_{>0}\}$. Thus
	\[
\cG^1=\frak{sl}(1,1)^{\widehat{ \ }}:=\frak{sl}(1,1)\otimes \C[t, t^{-1}]\oplus \C K.
	\]
	

Notice that any element $h\in \frak{h}$ such that $(h_\alpha | h)=0$ is a scalar multiple of $h_\alpha$. Hence, there exists no subalgebra $\cG^2$ of $\cG^0$ such that $\cG^0=\cG^1+\cG^2$ and $[\cG^1, \cG^2]=0$. In this case, we set
	\[
\cG^2=\fH_\beta := \C K\bigoplus_{k\in \Z \setminus \{0\}} (\C h_\beta\otimes t^k),
	\]
and we have that
	\[
\cG^0=\fH_\beta+\frak{sl}(1,1)^{\widehat{ \ }},\quad\text{and}\quad \frak{sl}(1,1)^{\widehat{ \ }} \cap \fH_\beta=\C K.
	\]
\end{enumerate}
Note that $\fH_\beta$ and $\fH_\alpha$ are both infinite-dimensional Heisenberg Lie algebras. Note also that the subalgebra of $\frak{sl}(1,1)^{\widehat{ \ }}$ generated by $h_\alpha\otimes t^{\pm 1}\C[t^{\pm 1}]$ is not a Heisenberg Lie algebra. Instead, it is a commutative algebra.
\end{example}

\section{Induced modules}

Let $P$ be a parabolic set of $\D$, and consider its associated decomposition
	\[
\cG=\cU^-\oplus (\cG^0+\cH)\oplus \cU,\quad\text{where}\quad \cP=(\cG^0+\cH)\oplus \cU
	\]
as in Theorem~\ref{thm:parabolic.structure}.   Any given $(\cG^0+\cH)$-module $N$ can be extended to a $\cP$-module defining $\cU N=0$.  Associated to $\cP$ and $N$, we define the induced module
	\[
M_{\cP}(N)=\Ind^\cG_{\cP} N.
	\]
If $N$ is simple, then $M_{\cP}(N)$ is called a \emph{generalized Verma type module} \cite{Fut97}. In this case $M_\cP(N)$  admits a unique maximal proper submodule, and thus a unique simple quotient, which will be denoted by $L_\cP(N)$. 

When $P$ is a parabolic partition we similarly define a \emph{Verma type module}
	\[
M_{\cB}(\lambda)=\Ind^\cG_{\cB} \C_{\lambda},
	\]
where $\lambda\in \cH^*$,  and $\C_\lambda$ denotes the one-dimensional $\cB$-module spanned by a vector $1_\lambda$ such that $(h+x)1_\lambda=\lambda(h)1_\lambda$, for all $h\in \cH$, $x\in \cN$. 

Our goal  is to study Verma type modules for $\cG$. The theory of standard Verma modules (i.e. those that are associated to parabolic partitions of the form $P(\dot{\S}, \dot{\S})$) is relatively well understood,  so we will focus our attention on non-standard Verma modules. Therefore, from now on we will assume that $X$ is a proper subset of $\dot{\S}$ (notice that $X$ can be empty). We will see that the Verma type module associated to the parabolic partition $P(\dot{\S}, X)$ is closely related to the generalized Verma type module associated to  the parabolic set $P=P(\dot{\S},\emptyset)\cup \langle\{-X\}\cup\{\pm \delta\}\rangle$. Before constructing such modules, let us first give a more explicit description of the Lie subalgebras associated to the parabolic set $P$. 

Let
	\[
\fm_X=\fm^-\oplus \fh\oplus \fm^+,\quad\text{where}\quad \fm^{\pm}=\bigoplus_{\alpha\in \dot{\D}^\pm(X)} \g_\alpha.
	\]
Next, if we consider the nilradical
	\[
\fu_X^\pm=\bigoplus_{\alpha\in \dot{\D}^\pm\setminus \dot{\D}(X)} \g_\alpha,
	\]
then
\begin{equation}\label{eq:defining.subalgebras}
\g = \fu_X^-\oplus \fm_X \oplus \fu_X^+, \quad  \cU_X^\pm = L(\fu_X^\pm), \quad \cG = \cU_X^-\oplus (\cG^0+ \cH) \oplus \cU_X^+.
\end{equation}

In order to describe the algebras $\cG^1$ and $\cG^2$ that appear in the decomposition of the Levi component $\cG^0$, we set $\fh_X :=\bigoplus_{\alpha\in \dot{\D}(X)}[\g_\alpha,\g_{-\alpha}]$ and we let $\fh_c\subseteq \fh$ be such that $\fh = \fh_X\oplus \fh_c$. Define
	\[
\fk_X = \fm^-\oplus \fh_X\oplus\fm^+\subseteq \fm_X.
	\]
Then
	\[
\cG^1=L(\fk_X)\oplus \C K,
	\]
and 
	\[
\cG^2 := \fH_X= (\fh_c\otimes t\C[t])\oplus \C K\oplus (\fh_c\otimes t^{-1}\C[t^{-1}])
	\] 
is a Heisenberg Lie algebra. Further, consider the decomposition
	\[
\fH_X=\fH_X^-\oplus \C K\oplus \fH_X^+,\quad \text{where}\quad \fH_X^\pm:= (\fh_c\otimes t^{\pm 1}\C[t^\pm]).
	\]	

Set
	\[
\cM_X=L(\fm_X) + \cH.
	\]
Then 
	\[
\cM_X=\cM_X^-\oplus \cH\oplus \cM_X^+,\quad \text{where}\quad \cM_X^\pm=(\fm_X\otimes t^\pm \C[t^\pm])\oplus \fm^\pm.
	\]
In other words, we have that
	\[
\cM_X=\cG^0+\cH=\cG^0\oplus \C d\oplus \fh_X^\perp,
	\] 
and $\cM_X^+$ is nothing but the standard Borel subalgebra of $\cM_X$. Moreover, we have 
	\[
\cG=\cU_X^-\oplus \cM_X \oplus \cU_X^+\quad \text{and}\quad  \cP_P^\pm=\cM_X\oplus \cU_X^\pm.
	\]

Finally, define
	\[
\cK_X=\cG^1+\cH= (L(\fk_X)\oplus \C K\oplus \C d)\oplus \fh_X^\perp,
	\]
and consider the decomposition
	\[
\cK_X=\cK_X^-\oplus \cH\oplus \cK_X^+,\quad \text{where}\quad  \cK_X^\pm=(\fk_X\otimes t^\pm\C[t^\pm])\oplus \fm^\pm.
	\]
Observe that only difference between $\cM_X$ and $\cK_X$ is that $\fH\subseteq \cM_X$ while $\fH \nsubseteq \cK_X$. Precisely, $L(\fh_X) \subseteq \cK_X$, but $(\fh_c\otimes t^\pm\C[t^\pm])\nsubseteq \cK_X$, and
	\[
\cM_X=\fH_X^-\oplus \cK_X\oplus \fH_X^+.
	\]
The subalgebra $\cM_X$ is naturally related to the Borel  subalgebra induced by the parabolic partition  $P(\dot{\S}, X)$. Namely, 
	\[
\cN_{P(\dot{\S}, X)}^\pm=\bigoplus_{\alpha\in \pm P(\dot{\S}, X)}\cG_\alpha, \quad \cB_{P(\dot{\S}, X)}^\pm=\cH \oplus \cN_{P(\dot{\S}, X)}^\pm\quad\text{and }\quad \cN_{P(\dot{\S}, X)}^\pm =\cM_X^\pm\oplus \cU_X^\pm. 
	\]
Moreover,
	\[
[\cB_{P(\dot{\S}, X)}^\pm, \cU_X^\pm]\subseteq \cU_X^\pm,\quad \text{and}\quad [\cM_X, \cU_X^\pm]\subseteq \cU_X^\pm.
	\]

From now on, we fix $X\subsetneqq \dot{\S}$ (and hence the associated parabolic partition $P(\dot{\S}, X)$ and the associated parabolic set $P=P(\dot{\S},\emptyset)\cup \langle\{-X\}\cup\{\pm \delta\}\rangle$) and we will drop $X$ and $P$ from the above notation (e.g. $\cP_P$, $\cK_X$,  etc., will be shorted by $\cP$, $\cK$, etc.). In the next sections that follow we will see that induced modules associated to these two sets are closely related. 

\subsection{Generalities on Verma type modules}\label{subsec:Verma.type.mod}
Let $\lambda\in \cH^*$,  and let
	\[
M_\cB(\lambda)={\rm Ind}^{\cG}_{\cB}(\C_\lambda)
	\]
be its corresponding Verma type module.  Set
	\[
M_\cM(\lambda)=\Ind_{\cM^+\oplus \cH}^\cM \C_\lambda,\quad \text{and}\quad M_\cK(\lambda)=\Ind_{\cK^+\oplus \cH}^\cK \C_\lambda.
	\]
It is standard to see that the modules $M_\cB(\lambda)$, $M_\cM(\lambda)$ and $M_\cK(\lambda)$ admit a unique simple quotient. Let $L_\cB(\lambda)$, $L_\cM(\lambda)$ and $L_\cK(\lambda)$ denote such simple quotients. 

If $F$ is a nonzero $\cM$-module, then we may consider $F$ as a $\cP$-module with trivial action of $\cU^+$. Then we define
	\[
M_\cP(F)=\Ind^{\cG}_{\cM} (F).
	\]
If $F$ is simple, the module $M_\cP(F)$ is nothing but the generalized Verma type module associated to $\cP$ and $F$. Our main goal is to find necessary and sufficient conditions for irreducibility of $M_\cP(F)$, for the case $F$ is a simple subquotient of $M_{\cM}(\lambda)$. As a corollary we give necessary and sufficient conditions for irreducibility of $M_\cB(\lambda)$.

In what follows, for a subalgebra $\cA$ of $\cG$, we define 
	\[
\D(\cA)=\{\alpha\in \D\mid \cG_\alpha\subseteq \cA\},
	\]
and we let
	\[
Q(\cA)=\{\sum_{i=1}^n k_i\alpha_i\mid k_i\in \Z_{\geq 0},\ \alpha_i\in \D(\cA)\}
	\] 
be the monoid in $\cH^*$ generated by $\D(\cA)$.

\begin{prop}\label{prop:structure.MB(lambda)}
The following standard conditions hold.
\begin{enumerate}
\item \label{prop.basic.prop.verm.mod.1} $\dim M_\cB(\lambda)_\lambda = 1$.
\item \label{prop.basic.prop.verm.mod.2} $M_\cB(\lambda)$ is a free $\bU(\cN^-)$-module generated by the vector $1 \otimes 1_\lambda$.
\item \label{prop.basic.prop.verm.mod.3} If $M_\cB(\lambda)_\mu\neq 0$, then $\dim M_\cB(\lambda)_\mu < \infty$ if and only if $\mu\in \{\lambda - \gamma \mid \gamma\in Q(\cM^+),\ n\in \Z_{> 0}\}$.
\end{enumerate}
\end{prop}
\begin{proof}
It follows from PBW Theorem, along with the fact that $\cN^+ = \cM^+\oplus \cU^+$, $\cM^+$ lies in a standard Borel subalgebra of $\cM$, and 
	\[
\D(\cU^+) = \{\alpha+n\delta\mid \alpha\in \dot{\D}^+(\dot{\S})\setminus \dot{\D}^+(X),\ n\in \Z\}.
	\]
\end{proof}

The analog of Proposition~\ref{prop:structure.MB(lambda)} for the modules $M_\cM(\lambda)$ and $M_\cK(\lambda)$ for part \eqref{prop.basic.prop.verm.mod.1} and \eqref{prop.basic.prop.verm.mod.2} is obvious. However, notice that part \eqref{prop.basic.prop.verm.mod.3} changes. Namely, the weight spaces of $M_\cM(\lambda)$ and $M_\cK(\lambda)$ will always be finite-dimensional.

Recall the Heisenberg Lie algebra associated to $\fh_c$:
	\[
\fH_X=\C K\oplus (\fh_c\otimes t\C[t])\oplus (\fh_c\otimes t^{-1}\C[t^{-1}]),
	\]
with decomposition
	\[
\fH_X=\fH_X^-\oplus \C K\oplus \fH_X^+,\quad\text{where}\quad \fH_X^\pm:= (\fh_c\otimes t^\pm\C[t^\pm]).
	\]	
For $\xi \in (\C K)^*$, let $M_{\fH_X}(\xi)$ denote the Verma  $\fH_X$-module associated to $\xi$, that is,
	\[
M_{\fH_X}(\xi)=\Ind^{\fH_X}_{\C K \oplus \fH_X^+} \C_\xi,
	\]
where $\C_\xi$ is spanned by $1_\xi$, $K 1_{\xi}=\xi(K) 1_{\lambda}$ and $x 1_{\xi}=0$ for all $x\in  \fH_X^+$.
	
The following is standard.
	
\begin{lem}\label{lem:irre.of.Heisenberg.modules}
The module $M_{\fH_X}(\xi)$ is simple if and only if $\xi(K)\neq 0$.
\end{lem}

\section{Non-isotropic case} 
\subsection{Criterion of simplicity}\label{sec:main.sec}
Using the notation from Theorem~\ref{thm:descr.of.G_0}, in this section we assume that no component $\g^i$ of $\cG^1$ is isomorphic to $\frak{sl}(1,1)$, that is:
\begin{equation*}
X \text{ does not have a connected component consisting of only one isotropic odd root. }
\end{equation*} 
Then the non-degeneracy of $(\cdot |\cdot )$ implies that we can choose the subalgebra $\fh_c$ such that $\fh_c = \{h\in \fh\mid 0=(h|h_\alpha)=\alpha(h), \text{ for all } \alpha\in \dot{\D}(X)\}$. (Note that if $X$ consists of only one isotropic root, then $\fh_X\subseteq \fh_c$, and $\fh_c\neq \fh$, since $(\cdot |\cdot)$ is non-degenerate). Then in this case we have that
	\[
[\fH_X, \cG^1]=0\Rightarrow [\fH_X, \cK_X^\pm]=0
	\]
and hence
\begin{equation}\label{eq:Enveloping.of.K}
\bU(\cM_X)=\bU(\fH_X^-)\bU(\cK_X)\bU(\fH_X^+).
\end{equation}
	

\begin{rem}
Consider $M_\cK(\lambda)\subseteq M_\cM(\lambda)$ in the obvious way. Since $[\fH_X^+, \cK^-]=0$ and $\fH_X^+ \cdot 1_\lambda=0$, we have 
	\[
\fH_X^+\cdot M_\cK(\lambda) = \fH_X^+\cdot \bU(\cK^-)\otimes_{\bU(\cM^+\oplus \cH)} \C_\lambda=\bU(\cK^-)\otimes_{\bU(\cM^+\oplus \cH)} \fH_X^+\cdot \C_\lambda=0.
	\] 
In particular, $\fH_X^+$ acts trivially on $M_\cK(\lambda)$ (and hence on any subquotient of it).
\end{rem}

\begin{prop}\label{prop:inducing.from.K.to.M}
Suppose $\lambda(K)\neq 0$. If a $\cK$-module $F$ is a simple subquotient of $M_{\cK}(\lambda)$, and we let $\fH_X^+\cdot F=0$, then the $\cM$-module $\Ind^{\cM}_{\cK\oplus \fH_X^+} F$ is a simple subquotient of $M_{\cM}(\lambda)$.
\end{prop}
\begin{proof}
First notice that, if we let $\fH_X^+$ act trivially on any $\cK$-module $V$, then, it follows from \eqref{eq:Enveloping.of.K} that
\begin{align*}
\bU(\cM)\otimes_{\bU(\cK\oplus \fH_X^+)} V & = \bU(\fH_X^-)\bU(\cK)\bU(\fH_X^+)\otimes_{\bU(\cK\oplus \fH_X^+)}V \\
& = \bU(\fH_X^-) \otimes_{\bU(\cK\oplus\fH_X^+)}V 
\end{align*}
as $\cM$-modules. Thus $\bU(\cM)\otimes_{\bU(\cK\oplus \fH_X^+)} V$ is a free $\bU(\fH_X^-)$-module. Let now $N\subseteq M\subseteq M_\cK(\lambda)$ be $\cK$-submodules such that $M/N\cong F$. Consider the two $\cM$-submodules of $M_\cM(\lambda)$ given by
	\[
{\bar M}:=\bU(\cM) \otimes_{\bU(\cK \oplus \fH_X^+)} M,\quad {\bar N}:=\bU(\cM) \otimes_{\bU(\cK \oplus \fH_X^+)} N.
	\]
It is easy to see that 
	\[
{\bar M}/{\bar N}\cong \bU(\cM) \otimes_{\bU(\cK \oplus \fH_X^+)} (M/N)\cong \bU(\cM) \otimes_{\bU(\cK \oplus \fH_X^+)} F
	\] 
as $\cM$-modules. 

It remains to prove that $\bU(\cM) \otimes_{\bU(\cK \oplus \fH_X^+)} F$ is a simple $\cM$-module. Since $F$ is a simple subquotient of $M_\cK(\lambda)$, there is $\mu\in \lambda + Q(\cM)$ such that $F\cong L_\cK(\mu)$. Let $u\in \bU(\cM) \otimes_{\bU(\cK \oplus \fH_X^+)} L_\cK(\mu)$. Then
	\[
u=\sum_{i,j} a_{i,j}u_i\otimes_{\bU(\cK \oplus \fH_X^+)} v_j,
	\]
where $\{u_i\}_{i\in I}$ is a basis for $\bU(\fH_X^-)$, $\{v_j\}_{j\in J}$ is a basis for $L_\cK(\mu)$, and $a_{i,j}\in \C$. Let $v_m$ be the minimal weight vector for which $a_{i,m}\neq 0$, for some $i\in I$, and fix $\ell\in I$ such that $a_{\ell,m}\neq 0$. Notice that $v_m\in L_\cK(\mu)_{\mu-\eta}$, for some $ \eta\in Q(\cM^+)$.  Since $a_{\ell,m}\neq 0$, and $\{v_j\}$ is linearly independent, we have that $0\neq \sum_{j\in J'} a_{\ell,j}v_j$, where $J'$ is the set of indexes for which $v_j\in L_\cK(\mu)_{\mu-\eta}$. Since $L_\cK(\mu)$ a simple $\cK$-module, there exists $w\in \bU(\cK)_\eta\subseteq \bU(\cK^+)$, such that $w(\sum_{j\in J'} a_{\ell,j}v_j) :=v_\mu$ is a highest weight vector of $L_\cK(\mu)$. Moreover, due to the choice of $v_m$, we have that $w v_j=0$, for all $j\notin J'$, and $w v_j \in \C v_\mu$ for all $j\in J'$. Hence, since $[\cK^+, \fH_X^-]=0$, we obtain
\begin{align*}
wu & = w\left(\sum_{i,j} a_{i,j}u_i\otimes_{\bU(\cK \oplus \fH_X^+)} v_j\right) \\
& = u_\ell \otimes_{\bU(\cK \oplus \fH_X^+)}  w \left(\sum_{j} a_{\ell,j}v_j \right)   + \left(\sum_{i\neq \ell}u_i\right)\otimes_{\bU(\cK \oplus \fH_X^+)}  \left(\sum_{j} a_{i,j} w v_j\right) \\
& = u_\ell\otimes_{\bU(\cK \oplus \fH_X^+)} v_\mu + \sum_{i\neq \ell} \beta_i u_i\otimes_{\bU(\cK \oplus \fH_X^+)} v_\mu\neq 0
\end{align*}
is an element in $\bU(\fH_X^-)\otimes_{\bU(\cK \oplus \fH_X^+)} v_\mu$. But, by Lemma~\ref{lem:irre.of.Heisenberg.modules}, the latter is a simple $\bU(\fH_X)$-module if and only if $\mu(K)\neq 0$, which is the case as $\mu\in \lambda+Q(\cM)$, $\lambda(K)\neq 0$, and $\gamma(K)=0$ for all $\gamma\in Q(\cM)$. Therefore, $v_\mu\in \bU(\fH_X)wu\subseteq \bU(\cM)u$, and the result follows.
\end{proof}

\begin{cor}\label{cor:inducing.from.K.to.M}
Suppose $\lambda(K)\neq 0$. Then any simple subquotient of $M_\cM(\lambda)$ is of the form $\Ind^{\cM}_{\cK\oplus \fH_X^+} L_\cK(\mu)$ for some $\mu\in \lambda+ Q(\cM)$.
\end{cor}
\begin{proof}
Any simple subquotient of $M_\cM(\lambda)$ is of the form $L_\cM(\mu)$, for some $\mu\in \lambda+Q(\cM)$. By the proof of Proposition~\ref{prop:inducing.from.K.to.M}, $\Ind^\cM_{\cK\oplus \fH_X^+}L_\cK(\mu)$ is simple. Then $L_\cM(\mu)\cong \Ind^\cM_{\cK\oplus \fH_X^+}L_\cK(\mu)$.
\end{proof}

Recall that
	\[
\fu^\pm:= \fu_X^\pm=\bigoplus_{\alpha\in \dot{\D}^\pm\setminus \dot{\D}(X)} \g_\alpha.
	\]

\begin{lem}\label{lem:tec.1}
If $\alpha\in \D(\fu^-)$ is a nonsimple root and $y\in \g_{-\alpha}$ is nonzero, then there exist $z\in \g_\gamma$ with $\gamma\in \D(\fn^+)$, such that $0\neq [z,y]\in \fu^-$.
\end{lem}
\begin{proof}
 The proof is done by induction on the rank $\rho$ of $\g$. If $\rho=1$, then $X=\emptyset$, $\fu^\pm=\fn^\pm$ and hence the result follows.

Assume $\rho>1$ and fix $\dot{\S}=\{\alpha_i\}_{i=1}^\rho$, a distinguished set of simple roots. Recall the distinguished Dynkin diagram associated to this choice and the enumeration of such a diagram (see \cite[pg.~343 - 360]{FSS00}). Write $\alpha=-\sum n_i\alpha_i$, and define $\Supp(\alpha):=\{\alpha_i\in \dot{\S}\mid n_i\neq 0\}$. If $\Supp (\alpha)\neq \dot{\S}$, then $\alpha$ belongs to a subalgebra of $\g$ with smaller rank, and the result follows by induction. Suppose then $\Supp (\alpha)=\dot{\S}$, that is, $n_i>0$ for all $i$. Now the proof can be completed case by case. We leave the details for the reader.
\details{
\subsection*{Case $A(m-1,n-1)=\frak{sl}(n,m)$.}
In this case $\dot{\S}=\{\alpha_1,\ldots,\alpha_n,\ldots, \alpha_{m+n-1}\}$, and the only possibility is
	\[
\alpha=-\alpha_1-\cdots-\alpha_{m+n-1}=-(\delta_1-\varepsilon_m)
	\]
If $X$ is not a string (i.e. its roots are not all connected in the distinguished Dynkin diagram), then $\alpha_j\notin X$ for some $2\leq j\leq m+n-1$. In particular, $\alpha+\alpha_1=-(\delta_2-\varepsilon_m)\in \D(\fu^-)$. On the other hand, if $X$ is a string, then since we are assuming $X\neq \dot{\S}$, either $\alpha_1\notin X$ or $\alpha_{m+n-1}\notin X$. Thus, either 
\begin{align*}
\alpha+\alpha_{m+n-1}= & -\alpha_1-\cdots-\alpha_{m+n-2}=-(\delta_1-\varepsilon_{m-1})\in \D(\fu^-),\text{ if } \alpha_1\notin X \text{ or }\\
\alpha+\alpha_1=& -\alpha_2-\cdots-\alpha_{m+n-1}=-(\delta_2-\varepsilon_m)\in \D(\fu^-)\text{ if }\alpha_{m+n-1}\notin X.
\end{align*}

\subsection*{Case $B(m,n)=\frak{osp}(2m+1, 2n)$.}
In this case $\dot{\S}=\{\alpha_1,\ldots,\alpha_n,\ldots, \alpha_{m+n}\}$, and the possibilities are
\begin{align*}
\alpha= & -\alpha_1-\cdots-\alpha_{l-1}-2\alpha_l-\cdots -2\alpha_{n+m}=-(\delta_1+\delta_l),\quad 1< l\leq n\\
\alpha=& -2\alpha_1-\cdots -2\alpha_{n+m}=-2\delta_1 \\
\alpha= & -\alpha_1-\cdots-\alpha_{n+i-1}-2\alpha_{n+i}-\cdots -2\alpha_{n+m}=-(\delta_1+\varepsilon_i),\quad 1\leq i\leq m\\
\alpha=& -\alpha_1-\cdots -\alpha_{n+m}=-\delta_1 
\end{align*}
For the three first cases, we have
\begin{align*}
\alpha+\alpha_l= & -\alpha_1-\cdots-\alpha_{l}-2\alpha_{l+1}-\cdots -2\alpha_{n+m}=-(\delta_1+\delta_{l+1}) \in \D(\fu^-)\text{ if } 1\leq l\leq n-1, \\
\alpha+\alpha_n= & -\alpha_1-\cdots-\alpha_{n}-2\alpha_{n+1}-\cdots -2\alpha_{n+m}=-(\delta_1+\varepsilon_1) \in \D(\fu^-)\text{ if } l = n; \\
\alpha+\alpha_1= & -\alpha_1-2\alpha_2-\cdots -2\alpha_{n+m}= -(\delta_1+\delta_2)\in \D(\fu^-); \\
\alpha+\alpha_{n+i}= & -\alpha_1-\cdots-\alpha_{n+i}-2\alpha_{n+i+1}-\cdots -2\alpha_{n+m}=-(\delta_1+\varepsilon_{i+1})\in \D(\fu^-)\text{ if } 1\leq i< m,\\
\alpha+\alpha_{n+m}= & -\alpha_1-\cdots-\alpha_{n+m}=-\delta_1\in \D(\fu^-)\text{ if } i= m;
\end{align*}
and the result follows since $X\neq \dot{\S}$. 

For the last, notice that if $X$ is not a string, then $\alpha_j\notin X$ for some $2\leq j\leq m+n$. In particular, 
	\[
\alpha+\alpha_1= -\alpha_2-\cdots -\alpha_{n+m}= -\delta_2\in \D(\fu^-).
	\]
On the other hand, if $X$ is a string, then since we are assuming $X\neq \dot{\S}$, either $\alpha_1\notin X$ or $\alpha_{n+m}\notin X$. In any case, we have that
\begin{align*}
\alpha+\alpha_{n+m}= & -\alpha_1-\cdots -\alpha_{n+m-1}= -(\delta_1-\varepsilon_{m})\in \D(\fu^-)\text{ if } \alpha_1\notin X,\text{ or } \\
\alpha+\alpha_1= & -\alpha_2-\cdots -\alpha_{n+m}= -\delta_2\in \D(\fu^-)\text{ if } \alpha_{n+m}\notin X.
\end{align*}

\subsection*{Case $B(0,2n)=\frak{osp}(1, 2n)$.}
In this case $\dot{\S}=\{\alpha_1,\ldots, \alpha_{n}\}$, and the possibilities are
\begin{align*}
\alpha= & -\alpha_1-\cdots-\alpha_{l-1}-2\alpha_l-\cdots -2\alpha_n=-(\delta_1+\delta_l), \quad 1\leq l \leq n, \\
\alpha=& -2\alpha_1-\cdots -2\alpha_{n}=-2\delta_1 \\
\alpha=& -\alpha_1-\cdots -\alpha_{n}=-\delta_1 
\end{align*}
For the two first cases, we have
\begin{align*}
\alpha+\alpha_l= & -\alpha_1-\cdots-\alpha_{l}-2\alpha_{l+1}-\cdots -2\alpha_n=-(\delta_1+\delta_{l+1}) \in \D(\fu^-)\text{ if } 1\leq l<n, \\
\alpha+\alpha_n= & -\alpha_1-\cdots-\alpha_n=-\delta_1 \in \D(\fu^-)\text{ if } l=n; \\
\alpha+\alpha_1= & -\alpha_1-2\alpha_2-\cdots -2\alpha_{n}= -(\delta_1+\delta_2)\in \D(\fu^-); \\
\end{align*}
and the result follows since $X\neq \dot{\S}$.

For the last case, If $X$ is not a string, then $\alpha_j\notin X$ for some $2\leq j\leq n$. In particular, 
	\[
\alpha+\alpha_1= -\alpha_2-\cdots -\alpha_{n}= -\delta_2\in \D(\fu^-).
	\]
If $X$ is a string, then since we are assuming $X\neq \dot{\S}$, either $\alpha_1\notin X$ or $\alpha_n\notin X$. Hence, either
\begin{align*}
\alpha+\alpha_{n}= & -\alpha_1-\cdots -\alpha_{n-1}= -(\delta_1-\delta_n) \in \D(\fu^-)\text{ if }\alpha_1\notin X\\
\alpha+\alpha_{1}= & -\alpha_2-\cdots -\alpha_n= -\delta_2  \in \D(\fu^-)\text{ if }\alpha_n \notin X.
\end{align*}

\subsection*{Case $C(n+1)=\frak{osp}(2, 2n)$.}
In this case $\dot{\S}=\{\alpha_1,\ldots, \alpha_{n+1}\}$, and the possibilities are
\begin{align*}
\alpha= & -\alpha_1-\cdots-\alpha_k-2\alpha_{k+1}-\cdots-2\alpha_n-\alpha_{n+1}=
-(\varepsilon+\delta_k),\quad 1\leq k< n \\
\alpha=& -\alpha_1-\cdots -\alpha_{n+1}=-(\varepsilon+\delta_n)
\end{align*}
For the first equation we get
\begin{align*}
\alpha+\alpha_{k+1}= & -\alpha_1- \cdots-\alpha_{k+1}-2\alpha_{k+2}-\cdots-2\alpha_n-\alpha_{n+1}= -(\varepsilon+\delta_k)\in \D(\fu^-) \text{ if } k\leq n-2 \\
\alpha+\alpha_{n}= & -\alpha_1- \cdots-\alpha_{n+1}= -(\varepsilon+\delta_n)\in \D(\fu^-) \text{ if } k=n-1,
\end{align*}
and the result follows, since $X\neq\dot{\S}$.

For the last case, if $X$ is not a string, then $\alpha_j\notin X$ for some $2\leq j\leq n+1$. In particular, 
	\[
\alpha+\alpha_1= -\alpha_2-\cdots -\alpha_{n+1}= -(\delta_1+\delta_n)\in \D(\fu^-).
	\]
If $X$ is a string, then since we are assuming $X\neq \dot{\S}$, either $\alpha_1\notin X$ or $\alpha_{n+1}\notin X$. Hence, either
\begin{align*}
\alpha + \alpha_{n+1} = & -\alpha_1 - \cdots  - \alpha_{n} = -(\varepsilon - \delta_n) \in \D(\fu^-)\text{ if }\alpha_1\notin X \\
\alpha + \alpha_1 = & -\alpha_2 - \cdots  - \alpha_{n+1} = -(\delta_1 - \delta_n) \in \D(\fu^-) \text{ if }\alpha_{n+1}\notin X.
\end{align*}

\subsection*{Case $D(m,n)=\frak{osp}(2m, 2n)$.}
In this case $\dot{\S}=\{\alpha_1,\ldots,\alpha_n,\ldots, \alpha_{n+m}\}$, and the possibilities are
\begin{align*}
\alpha= & -\alpha_1-\cdots-\alpha_{l-1}-2\alpha_l-\cdots-2\alpha_{n+m-2}-\alpha_{n+m-1}-
\alpha_{n+m} \\
= & -(\delta_1+\delta_l),\quad 1\leq l\leq n \\
\alpha=& -2\alpha_1-\cdots -2\alpha_{n+m-2}-\alpha_{n+m-1}-\alpha_{n+m}=-2\delta_1 \\
\alpha=& -\alpha_1-\cdots -\alpha_{n+i-1}-2\alpha_{n+i}-\cdots -2\alpha_{n+m-2}-\alpha_{n+m-1}-\alpha_{n+m} \\
=& -(\delta_1+\varepsilon_i),\quad 1\leq i<m-1 \\
\alpha=& -\alpha_1-\cdots -\alpha_{n+m-2}-\alpha_{n+m-1}-\alpha_{n+m}=-(\delta_1+\varepsilon_{m-1}).
\end{align*}
For the three first cases, we have
\begin{align*}
\alpha+\alpha_l= & -\alpha_1-\cdots-\alpha_l-2\alpha_{l+1}-\cdots-2\alpha_{n+m-2}-\alpha_{n+m-1}-
\alpha_{n+m} \\
= & -(\delta_1+\delta_{l+1})\in \D(\fu),\quad 1\leq l< n \\
\alpha+\alpha_n=& -\alpha_1-\cdots -\alpha_{n}-2\alpha_{n+1}-\cdots -2\alpha_{n+m-2}-\alpha_{n+m-1}-\alpha_{n+m} \\
=& -(\delta_1+\varepsilon_1)\in \D(\fu),\quad l=n \\
\alpha+\alpha_1=& -\alpha_1-2\alpha_2-\cdots -2\alpha_{n+m-2}-\alpha_{n+m-1}-\alpha_{n+m}=-(\delta_1+\delta_{2})\in \D(\fu) \\
\alpha+\alpha_{n+i}=& -\alpha_1-\cdots -\alpha_{n+i}-2\alpha_{n+i+1}-\cdots -2\alpha_{n+m-2}-\alpha_{n+m-1}-\alpha_{n+m} \\
=& -(\delta_1+\varepsilon_{i+1})\in \D(\fu),\quad 1\leq i<m-1.
\end{align*}
and the result follows, since $X\neq \dot{\S}$.

For the last case. If $X$ is not a string, then $\alpha_j\notin X$ for some $2\leq j\leq m+n$. In particular, 
	\[
\alpha+\alpha_1=-\alpha_2-\cdots -\alpha_{n+m-2}-\alpha_{n+m-1}-\alpha_{n+m}=-(\delta_2+\varepsilon_{m-1})\in \D(\fu)
	\]
If $X$ is a string, then since we are assuming $X\neq \dot{\S}$, either $\alpha_1\notin X$ or $\alpha_{m+n}\notin X$. Hence, either
\begin{align*}
\alpha+\alpha_{n+m}=& -\alpha_1-\cdots -\alpha_{n+m-2}-\alpha_{n+m-1}=-(\delta_1-\varepsilon_{m})\in \D(\fu), \text{ if }\alpha_1\notin X,\text{ or } \\
\alpha+\alpha_1=& -\alpha_2-\cdots -\alpha_{n+m-2}=-(\delta_2+\varepsilon_{m-1})\in \D(\fu), \text{ if } \alpha_{m+n}\notin X.
\end{align*}

\subsection*{Case $F(4)$.}
In this case $\dot{\S}=\{\alpha_1,\alpha_2,\alpha_3,\alpha_4\}$, and the possibilities are
\begin{align*}
\alpha= & - 2\alpha_1-3\alpha_2-2\alpha_3-\alpha_4 \\
\alpha= & - \alpha_1-2\alpha_2-\alpha_3-\alpha_4 \\
\alpha= & - \alpha_1-2\alpha_2-2\alpha_3-\alpha_4 \\
\alpha= & - \alpha_1-3\alpha_2-2\alpha_3-\alpha_4 \\
\alpha= & - \alpha_1-\alpha_2-\alpha_3-\alpha_4.
\end{align*}
For the four first cases, we have
\begin{align*}
\alpha+\alpha_1= & - \alpha_1-3\alpha_2-2\alpha_3-\alpha_4\in \D(\fu^-) \\
\alpha+\alpha_2= & -\alpha_1-\alpha_2-\alpha_3-\alpha_4\in \D(\fu^-) \\
\alpha+\alpha_3= & -\alpha_1-2\alpha_2-\alpha_3-\alpha_4\in\D(\fu^-) \\
\alpha+\alpha_2= & -\alpha_1-2\alpha_2-2\alpha_3-\alpha_4\in\D(\fu^-).
\end{align*}
and the result follows from the fact that $X\neq \dot{\S}$.

For the last case, if $X$ is not a string, then $\alpha_j\notin X$ for some $2\leq j\leq 4$. In particular, 
	\[
\alpha+\alpha_1= -\alpha_2-\alpha_3-\alpha_4\in \D(\fu^-).
	\]
If $X$ is a string, then since we are assuming $X\neq \dot{\S}$, either $\alpha_1\notin X$ or $\alpha_4\notin X$. Hence, either
\begin{align*}
\alpha+\alpha_4 = & -\alpha_1-\alpha_2-\alpha_3\in \D(\fu^-),\text{ if } \alpha_1\notin X, \text{ or } \\
\alpha+\alpha_1= & -\alpha_2-\alpha_3-\alpha_4\in \D(\fu^-) \text{ if } \alpha_4\notin X.
\end{align*}

\subsection*{Case $G(3)$.}
In this case $\dot{\S}=\{\alpha_1,\alpha_2,\alpha_3\}$, and the possibilities are
\begin{align*}
\alpha= & - 2\alpha_1-4\alpha_2-2\alpha_3 \\
\alpha= & - \alpha_1-2\alpha_2-\alpha_3 \\
\alpha= & - \alpha_1-3\alpha_2-\alpha_3 \\
\alpha= & - \alpha_1-3\alpha_2-2\alpha_3 \\
\alpha= & - \alpha_1-4\alpha_2-2\alpha_3 \\
\alpha= & - \alpha_1-\alpha_2-\alpha_3.
\end{align*}
For the first five cases, we have
\begin{align*}
\alpha+\alpha_1= & - \alpha_1-4\alpha_2-2\alpha_3\in\D(\fu^-) \\
\alpha+\alpha_2= & - \alpha_1-\alpha_2-\alpha_3\in\D(\fu^-) \\
\alpha+\alpha_2= & - \alpha_1-2\alpha_2-\alpha_3\in\D(\fu^-) \\
\alpha+\alpha_3= & - \alpha_1-3\alpha_2-\alpha_3\in\D(\fu^-) \\
\alpha+\alpha_2= & - \alpha_1-3\alpha_2-2\alpha_3\in\D(\fu^-).
\end{align*}
and the result follows, since $X\neq \dot{\S}$.

For the last case, if $X$ is not a string, then $\alpha_2\notin X$. In particular, 
	\[
\alpha+\alpha_1= -\alpha_2-\alpha_3\in\D(\fu^-).
	\]
If $X$ is a string, then since we are assuming $X\neq \dot{\S}$, either $\alpha_1\notin X$ or $\alpha_3\notin X$. Hence, either
\begin{align*}
\alpha+\alpha_3= & -\alpha_1-\alpha_2\in\D(\fu^-)\text{ if }\alpha_1\notin X,\text{ or } \\
\alpha+\alpha_1= & -\alpha_2-\alpha_3\in\D(\fu^-) \text{ if }\alpha_3\notin X.
\end{align*}

\subsection*{Case $D(2,1;\alpha)$.}
In this case $\dot{\S}=\{\alpha_1,\alpha_2,\alpha_3\}$, and the possibilities are
\begin{align*}
\alpha= & - 2\alpha_1-\alpha_2-\alpha_3 \\
\alpha= & - \alpha_1-\alpha_2-\alpha_3.
\end{align*}
For the first case, we have that
	\[
\alpha+\alpha_1= - \alpha_1-\alpha_2-\alpha_3\in\D(\fu^-),
	\]
and the result follows from the fact that $X\neq \dot{S}$.

For the last case, if $X$ is not a string, then $\alpha_2\notin X$. In particular, 
	\[
\alpha+\alpha_3= - \alpha_1-\alpha_2\in\D(\fu^-).
	\]
If $X$ is a string, then since we are assuming $X\neq \dot{\S}$, either $\alpha_1\notin X$ or $\alpha_3\notin X$. Hence, either
\begin{align*}
\alpha+\alpha_3= & - \alpha_1-\alpha_2\in\D(\fu^-)\text{ if } \alpha_1\notin X \\
\alpha+\alpha_2= & - \alpha_1-\alpha_3\in\D(\fu^-)\text{ if } \alpha_3\notin X.
\end{align*}
Thus the proof is complete.
}
\end{proof}

\begin{rem}
Although in this section we are assuming that the set $X$ does not have a connected component consisting of only one isotropic odd root, for the proof of Lemma~\ref{lem:tec.1} we do not use this fact.
\end{rem}

Now we reproduce the ordered basis for $L(\g)$ given in \cite{Cox94}. Let $\lambda,\mu\in \cH^*$, we say $\lambda>\mu$ if and only if $\lambda-\mu\in Q(\cN^+)=\sum_{\alpha\in P(\dot{\S}, X)} \Z_{\geq 0}\alpha$. For each $\alpha\in \dot{\D}^+(\dot{\S})$, one can find nonzero vectors $z_\alpha\in \g_\alpha$, $z_{-\alpha}\in \g_{-\alpha}$ and $h_\alpha\in \fh$ such that $[z_{\alpha}, z_{-\alpha}]=h_\alpha$. Each such triple generate a subalgebra of $\g$ which is isomorphic to either $\frak{sl}(2)$, $\frak{osp}(1,2)$ or $\frak{sl}(1,1)$. In particular, if we recall that $h_0$ denotes the identity matrix $I_{m,n}$ of $\fgl(m,n)$, then we have 
\begin{align*}
& C=\{z_\alpha, z_{-\alpha}\mid \alpha\in \dot{\D}^+(\dot{\S})\}\cup \{h_i:=h_{\alpha_i}\mid \alpha_i\in \dot{\S}\} \text{ is a basis of } \g\text{ if } \g\neq \fgl(m,n),\\
& C=\{z_\alpha, z_{-\alpha} \mid \alpha\in \dot{\D}^+(\dot{\S})\}\ \cup \{h_i:=h_{\alpha_i},\ h_0\mid \alpha_i\in \dot{\S}\} \text{ is a basis of } \fgl(m,n).
\end{align*}

Let $\{x_j\}_{j=1}^{\dim \fm^+}\subseteq C$ be a basis of $\fm^+$  and $\{x_j\}_{-\dim \fm^-}^{j=-1}\subseteq C$ be a basis of $\fm^-$. Let us now order $\fm$ as follows
\begin{enumerate}
\item If $x_i\in \g_{\beta_i}$ and $x_j\in \g_{\beta_j}$, with $\beta_i<\beta_j$, then $i<j$.
\item If $x_i\in \g_{\beta_i}$, then $x_{-i}\in \g_{-\beta_i}$.
\end{enumerate}
We define $x_i<x_j$ if $i<j$. Next let $\{y_j\}_{j=1}^{\dim \fu^+}\subseteq C$ be a basis of $\fu^+$ and $\{y_j\}_{-\dim \fu^-}^{j=-1}\subseteq C$ be a basis of $\fu^-$ and order these basis as above. Therefore we define a totally order on our basis $C$ by setting
	\[
\fu^-<\fm^-<h_i<h_j<\fm^+<\fu^+,
	\]
where $1\leq i<j\leq \dim \fh$. Notice that the order relation above is compatible with the partial order on $\g$ induced by $\dot{\D}^+(\dot{\S})$.

Let us now give an order for a basis of $L(\g)$. First consider the basis $B(L(\g))=\{ z(m)\mid z\in C, m\in \Z\}$ of $L(\g)$. Define a total order on $B(L(\g))$ by setting
	\[
z(m)<w(n),
	\]
if either $m<n$ or $m=n$ and $z<w$. 

Consider the multi-index notation: for $r\in \Z_{\geq 1}$, we set 
	\[
(\bi, \bm,\bp) = (i_1,\ldots, i_r, m_1,\ldots, m_r, p_1,\ldots, p_r)\in \Z^{3r},
	\]
and define
	\[
\bar{y}  =  y_{\bi, \bm,\bp} = y_{i_1}(m_1)^{p_1}\cdots y_{i_r}(m_r)^{p_r} \in \bU(\cU^-).
	\]
Since $B(\cU^-)=\{y_j(m)\mid m\in \Z\}$ is a basis for $\cU^-$, by PBW Theorem, $\bU(\cU^-)$ has a basis given by 
	\[
B(\bU(\cU^-)) = \{{\bar y}=y_{\bi, \bm,\bp} \mid  y_{i_j}(m_j)<y_{i_{j+1}}(m_{j+1}),\ y_{i_j}(m_j)\in B(\cU^-)\}.
	\]
We say that $\bar y=0$ if some $p_i<0$ and that $\bar y=1$, if $p_i=0$ for all $i=1\ldots, r$. For an element ${\bar z}=z_{i_1}(m_1)^{p_1} \cdots z_{i_r}(m_r)^{p_r}\in \bU(\cU)$, we define
\begin{align*}
&{\bar z}^{\hat j}=z_{i_1}(m_1)^{p_1}\cdots z_{i_j}(m_j)^{p_j-1}\cdots z_{i_r}(m_r)^{p_r} \\
& {\bar z}^{\hat{j}\hat{\ell}}=z_{i_1}(m_1)^{p_1}\cdots z_{i_j}(m_j)^{p_j-1}\cdots z_{i_\ell}(m_\ell)^{p_\ell-1}\cdots z_{i_r}(m_r)^{p_r}, \text{ if } j\neq \ell\\
& {\bar z}^{\hat{j}\hat{\ell}}=z_{i_1}(m_1)^{p_1}\cdots z_{i_j}(m_j)^{p_j-2}\cdots z_{i_r}(m_r)^{p_r},\text{ if } j=\ell.
\end{align*}

Let $\bU(\cU^-)_{(p)}$ be the subspace of $\bU(\cU^-)$ spanned by all ${\bar y}$ in $B(\bU(\cU^-))$ with degree $\leq p$ (here $\deg ({\bar y})= \sum p_i$). Set also $\bU(\cU^-)_{(p)}=0$, if $p<0$. It is clear that $\bU(\cU^-)=\sum_{p\geq 0} \bU(\cU^-)_{(p)}$. If $u\in \bU(\cU^-)_{(p)}$ and $u\notin \bU(\cU^-)_{(p-1)}$, then we say $u$ has degree $p$. If ${\bar y}\in \bU(\cU^-)_{(p)}$ and ${\bar y}'\in \bU(\cU^-)_{(q)}$, with $p<q$, then we say ${\bar y}<{\bar y}'$. Now we order the monomials in $\bU(\cU^-)_{(p)}$ reverse lexicographically (with respect to $(\bi,\bm,\bp)$). Therefore we have a total order on $\bU(\cU^-)$. With respect to such an order, we have for instance: 
	\[
y_{-2}(10)<y_{-1}(2)^3y_{-3}(4)^2<y_{-1}(2)y_{-3}(4)^4.
	\]

If $V$ is an $\cM$-module with totally ordered basis $\{v_j\}_{j\in J}$, then 
	\[
\Ind^{\cG}_{\cM} V\cong \bU(\cU^-)\otimes_{\C} V
	\]
has basis consisting of elements of the form ${\bar y}v_j$. Order this basis lexicographically: ${\bar y}v_i< {\bar y}'v_j$ if either ${\bar y}< {\bar y}'$ or if ${\bar y}={\bar y}'$ and $i<j$.

For a nonzero element
	\[
u=\sum a_{\bi,\bm,\bp}^i y_{\bi,\bm,\bp} v_i,\quad\text{where}\quad a_{\bi,\bm,\bp}^i\in \C,
	\]
we define
	\[
\LinSpan(u)=\Span\{y_{\bi,\bm,\bp} v_i \mid a_{\bi,\bm,\bp}^i \neq 0\}.
	\]

Before stating the next result, define
\begin{align*}
J^c= & J^c(z,\bar y)=\{1\leq j\leq r\mid [y_{i_j}, z]\in \fu^+\oplus \fm\} \\
J_-= & J_-(z,\bar y)=\{1\leq j\leq r\mid [y_{i_j}, z]\in \fu^-\} \\
J_-^l= & J_-^l(z,\bar y, e)=\{j\in J_-\mid y_{i_j}(m_j)\geq [z, y_{i_j}](e+m_j)\}\\
J_-^r= & J_-^r(z,\bar y, e)=\{j\in J_-\mid y_{i_j}(m_j)< [z, y_{i_j}](e+m_j)\}.
\end{align*}

The next result is an adaptation of \cite[Lemma~4.3]{Cox94} to the super setting.

\begin{lem}\label{lem:tec.2}
Suppose that $\lambda(K)\neq 0$ and let $N$ be a subquotient of $M_\cM(\lambda)$. Let $0\neq v \in N$, $z\in \g$, $e\in \Z$, and $\bar{y} = y_{i_1}(m_1)^{p_1}\cdots y_{i_r}(m_r)^{p_r}$, with $\deg{\bar y}\geq 1$. If either $z\in \fu^+$ and $e\ll 0$, or $e\gg 0$ and $z\in \fm^+$, then
\begin{align*}
z(e){\bar y}v & \equiv  \sum_{j\in J_-^l} (-1)^{S_{1_j}} p_j[z, y_{i_j}](m_j+e){\bar y}^{j}v +  \sum_{j\in J_-^r\cup J^c} (-1)^{S_{1_j}}  p_j {\bar y}^{j} [z, y_{i_j}](m_j+e) v \\ 
& + \sum_{j\in J_-^l\setminus \{1\}} (-1)^{S_{1_j}S_{2_j}S_{3_j}} \sum_{\eta=1}^{j-1} p_\eta p_j [y_{i_\eta},[z, y_{i_j}]](e+m_j+m_\eta){\bar y}^{\hat{j}\hat{\eta}}v \\
& + \sum_{j\in J_-^r\cup J^c \setminus \{r\}} (-1)^{S_{1_j}S_{2_j}S_{3_j}} \sum_{\eta=j+1}^{r} [[z, y_{i_j}],y_{i_\eta}](e+m_j+m_\eta){\bar y}^{\hat{j}\hat{\eta}}v \\
& +  \sum_{1\leq j\leq r} (-1)^{S_{1_j}S_{2_j}} \frac{p_j(p_j-1)}{2}[y_{i_j},[z, y_{i_j}]](2m_j+e){\bar y}^{{\hat j}{\hat j}}v  \quad \mod \bU(\cU^-)_{(p-2)}\otimes N.
\end{align*}
where $p= \sum p_i$, and $S_{i_j}$ depends on the parity of $y_{i_j}$ and $z$.
\end{lem}
\begin{proof}
Let $x^0=1$ for any $x\in \bU(\cG)$ and if the index set in the summation above is empty, then we take the summation to be zero.

Notice that
\begin{multline*}
z(e) y_{i_1}(m_1)^{p_1}\cdots y_{i_r}(m_r)^{p_r} \\
= \sum_{j=1}^{r} (-1)^{S_{1_j}}\sum_{\ell=0}^{p_j-1} y_{i_1}(m_1)^{p_1}\cdots y_{i_j}(m_j)^\ell [z, y_{i_j}](e+m_j) y_{i_j}(m_j)^{p_j-\ell-1}\cdots y_{i_r}(m_r)^{p_r}v \\
+  (-1)^{r|z|}y_{i_1}(m_1)^{p_1}\cdots y_{i_r}(m_r)^{p_r}z(e).
\end{multline*}
Hence, if either $z\in \fu^+$ and $e\ll 0$, or $e\gg 0$ and $z\in \fm^+$, then $z(e)v=0$, and we get that
\begin{align*}
z(e){\bar y}v & = \biggl( \sum_{j=1}^{r} (-1)^{S_{1_j}}\sum_{\ell=0}^{p_j-1} y_{i_1}(m_1)^{p_1}\cdots y_{i_j}(m_j)^\ell [z, y_{i_j}](e+m_j) y_{i_j}(m_j)^{p_j-\ell-1}\cdots y_{i_r}(m_r)^{p_r}v\biggr)  \\
& = \sum_{i\in J_-^l} (-1)^{S_{1_j}} \biggl(\sum_{\ell=0}^{p_j-1} [z, y_{i_j}](m_j+e){\bar y}^{\hat j}v \\ 
& + (-1)^{S_{2_j}} [y_{i_1}(m_1)^{p_1}\cdots y_{j}(m_j)^{\ell}, [z,y_{i_j}](m_j+e)] y_{j}(m_j)^{p_j-\ell-1}\cdots y_{I_r}(m_r)^{p_r} v\biggr) \\
& + \sum_{i\in J_-^r\cup J^c} (-1)^{S_{1_j}} \biggl(\sum_{\ell=0}^{p_j-1} {\bar y}^{\hat j}[z, y_{i_j}](m_j+e)v \\ 
& + (-1)^{S_{2_j}} y_{i_1}(m_1)^{p_1}\cdots y_{j}(m_j)^{\ell}[[z,y_{i_j}](m_j+e), y_{j}(m_j)^{p_j-\ell-1}\cdots y_{I_r}(m_r)^{p_r}]  v\biggr) \\
& \equiv  \sum_{j\in J_-^l} (-1)^{S_{1_j}} p_j[z, y_{i_j}](m_j+e){\bar y}^{j}v +  \sum_{j\in J_-^l} (-1)^{S_{1_j}S_{2_j}} \frac{p_j(p_j-1)}{2}[y_{i_j},[z, y_{i_j}]](2m_j+e){\bar y}^{{\hat j}{\hat j}}v \\ 
& + \sum_{j\in J_-^l\setminus \{1\}} (-1)^{S_{1_j}S_{2_j}S_{3_j}} \sum_{\eta=1}^{j-1} p_\eta p_j [y_{i_\eta},[z, y_{i_j}]](e+m_j+m_\eta){\bar y}^{\hat{j}\hat{\eta}}v \\
& + (-1)^{S_{1_j}}  \sum_{i\in J_-^r\cup J^c} p_j {\bar y}^{j} [z, y_{i_j}](m_j+e) {\bar x}  \\
& +  \sum_{j\in J_-^l\cup J^c} (-1)^{S_{1_j}S_{2_j}} \frac{p_j(p_j-1)}{2}{\bar y}^{{\hat j}{\hat j}}[[z, y_{i_j}],y_{i_j}](2m_j+e)v \\ 
& + \sum_{j\in J_-^r\cup J^c \setminus \{r\}} (-1)^{S_{1_j}S_{2_j}S_{3_j}} \sum_{\eta=j+1}^{r} [[z, y_{i_j}],y_{i_\eta}](e+m_j+m_\eta){\bar y}^{\hat{j}\hat{\eta}}v\quad \mod \bU(\cU^-)_{(p-2)}\otimes N.
\end{align*}
\end{proof}

Let $N$ be a simple subquotient of $M_\cM(\lambda)$. If $\lambda(K)\neq 0$, then by Corollary~\ref{cor:inducing.from.K.to.M}, $N$ is a free $\bU(\fH_X^-)$-module with basis $\{w_\xi\}_{\xi\in \Xi}$ consisting of weight vectors. Let $\{v_\eta\}_{\eta\in \Gamma}$ be a basis of weight vectors of $\bU(\fH_X^-)$. Then $B=\{v_\eta w_\xi\mid (\eta, \xi)\in \Gamma\times \Xi\}$ is a basis of weight vectors of $N$. Set $I := \Gamma\times \Xi$, and for any $i = (\xi, \eta)\in I$ we associate the vector $x_i = v_\eta w_\xi$. Since the set of weights of $N$ is contained in the coset $\lambda - Q(\cM^+)$, we can order the basis $B$ so that
\begin{equation}\label{cond.basis}
\text{ if } x_i \text{ has weight } \lambda_i \text{ and if } \lambda_i>\lambda_j, \text{ then } i>j.
\end{equation}

Now consider two finite sets $J$ and $S$, where $J\subseteq I$ and $\{m_r\}_{r\in S}\subseteq \Z$ such that $m_r\neq m_s$ if $r\neq s$. For a pair $(j,r)\in J\times S$, we pick $x_j=v_\eta w_\xi\in B$, and we choose $e\ll 0$, such that $h(e+m_r)$ commutes with $v_\eta$ (since no central term appear) for any $h\in \fh_X$. Then, we can write
	\[
h(e+m_r) x_j=\sum_{\xi'} a_{\xi'}v_\eta w_{\xi'}, \quad \text{ where}\quad a_{\xi'}\in \C.
	\]
Moreover, if $(i,s)\in J\times S$ is another pair, then for any $z\in \fh_X^\perp$ we can find $e\ll 0$ so that $z(e+m_s)x_i\in B$. In particular, we have the following lemma.

\begin{lem}\label{lem:tec.3}
With the above notation, if we choose $e\ll 0$, then
	\[
z(e+m_s)x_i\notin \Span\{ z(e+m_r)x_j, h(e+m_r)x_j\mid (i.s)\neq (j,r),\ (j,r)\in J\times S\}.
	\]
\end{lem}

Recall that $P=P(\dot{\S},\emptyset)\cup \langle\{-X\}\cup\{\pm \delta\}\rangle\subseteq \D$ is a parabolic set, 
	\[
\cG=\cU^-\oplus \cM \oplus \cU^+\quad\text{and}\quad  \cP=\cM\oplus \cU^+,
	\]
that for any $\cM$-module $F$ we have defined the generalized Verma type module $M_\cP(F):=\Ind^{\cG}_{\cM} (F)$, where $\cU^+\cdot F=0$, and that in this section we are assuming that:
\begin{equation*}
X \text{ does not have a connected component consisting of only one isotropic odd root. }
\end{equation*}

\begin{prop}\label{prop:inductive.argument}
Suppose that $\lambda(K)\neq 0$, and let $N$ be a subquotient of $M_\cM(\lambda)$ with a basis $B$ satisfying \eqref{cond.basis}. Then any non-zero submodule of $M_\cP(N)$ intersects $1\otimes N$ non-trivially.
\end{prop}
\begin{proof}
It is sufficient to prove that if $0\neq u\in M_\cP(N)_{\lambda-\beta}$ for some $\beta\in Q(\cN^+)\setminus Q(\cM^+)$, then there exists $\alpha\in Q(\fn^+)\setminus Q(\fm^+)$, $p\in \Z$ and $Y\in \bU(\cN^+)_{\alpha+p\delta}$ such that $Yu\neq 0$. With Lemmas~\ref{lem:tec.1}, \ref{lem:tec.2} and \ref{lem:tec.3} on hand, the proof follows the same steps as that of \cite[Proposition~4.5]{Cox94}. 

\details{
Let us start noticing that $\{y_{\bi,\bm,\bp} x_i\}$ is a lexicographically ordered basis for $M_\cP(N)$. Moreover, the filtration (by the degree) of $\bU(\cU^-)$ induces a filtration by degree on $M_\cP(N)\cong \bU(\cU^-)\otimes N$. Let $0\neq u\in M_\cP(N)_{\lambda-\beta}$ for some $\beta\in Q(\cN^+)\setminus Q(\cM^+)$. The fact that $\beta\in Q(\cN^+)\setminus Q(\cM^+)$ implies that $\deg u>0$, that is, the degree of a maximal monomial occurring in $u$ is $>0$ (in fact, the degree zero component of $M_\cP(N)$ is $N$ whose weights lie in $\lambda - Q(\cM^+)$). Fix a monomial with maximal degree $y_{\ba, \bb, \bc} x_d$ occurring in $u$. Notice that there may exist many of the same degree. In particular, $\deg y_{\ba, \bb,\bc}>0$.

Assume that $J_-(z, {\bar y})\neq \emptyset$, for some $z\in \g_\gamma$ with $\gamma\in \D(\fu^+)\cup \D(\fm^+)$, and some ${\bar y}$ with maximal degree occurring in $u$. For such a $z$,  consider $M_z$ to be the set of all ${\bar y}$ with maximal degree occurring in $u$ such that $J_-(z, {\bar y})\neq \emptyset$. By assumption, $M_z\neq \emptyset$. Among all $y_{i_j}$ factors that occur in the elements of $M_z$, let $y_{a_k'}$ be the one whose associated root $\alpha_{a_k'}$ is minimal. Fix some monomial ${\bar y}_{\min}=y_{\ba', \bb', \bc'}\in M_z$ that contains $y_{a_k'}$ as a component, and define $J_{\min}=J_-(z, {\bar y}_{\min})$ and $J_{\min}^c=J^c(z, {\bar y}_{\min})$. Notice that ${\bar y}_{\min}$ may occur more than once in $u$. Let $j$ be an index such that $u_{\ba', \bb', \bc'}^j\neq 0$ and fix $x_j=x_{\min}$. The idea is to reduce to the case where all components $y_{i_j}$ of any ${\bar y}\in M_z$ lie in a root space associated to a simple root.

By Lemma~\ref{lem:tec.2}, we have that 
	\[
z(e){\bar y}_{\min} x_{\min} \equiv \sum_{j\in J_{\min}} (-1)^{S_{1_j}}c_j[z, y_{a_j'}](e+b_j'){\bar y}_{\min}^{\hat j} x_{\min} \quad \mod \bU(\cU^-)_{(c'-1)}\otimes N,
	\]
where $c'=\deg {\bar y}_{\min}$, and either $e\ll 0$ and $z\in \g_\gamma$ for $\gamma\in \D(\fu^+)$ or $e\gg0$ and $z\in \g_\gamma$ for $\D(\fm^+)$. Since $J_{\min}\neq 0$, and we are choosing $e\ll 0$ or $e\gg0$, the element $[z, y_{a_j'}](e+b_j'){\bar y}_{\min}^{\hat j}$ lie in $B(\bU(\cU^-))$, and hence $S_1\neq 0$, where $S_1$ denotes the first summation above. Suppose now ${\bar y}x$, with ${\bar y}=y_{\bi, \bm,\bp}$ and $x=x_i$ is a different monomial occurring in $u$. Then
	\[
z(e){\bar y} x \equiv \sum_{j\in J_{\min}} (-1)^{S_{1_j}}p_j[z, y_{i_j}](e+m_j){\bar y}^{\hat j} x \quad \mod \bU(\cU^-)_{(p-1)}\otimes N,
	\]
where $p=\deg {\bar y}$, and $p\leq c'$, by construction. Let $T_1$ be the summation above. If $T_1\in \LinSpan(S_1)+ \bU(\cU^-)_{(p-1)}\otimes N$, then one must have $p=c'$, $\C [z, y_{i_j}]=\C [z, y_{a_l'}]$, $m_j=b_l$, ${\bar y}^{\hat j}={\bar y}_{\min}^{\hat l}$ and $x=x_{\min}$, since we are taking $e\ll 0$ or $e\gg0$.  This would imply ${\bar y}x=y_{\min} x_{\min}$, which is a contradiction. Therefore $T_1\notin \LinSpan(S_1)+ \bU(\cU^-)_{(p-1)}\otimes N$. This proves that $z(e)u$ has the same degree as $u$ and that it sends each monomial of maximal degree to a monomial that has higher weight. 

Since $|\dot{\D}|<\infty$, after applying the above argument finitely many times, we may assume that $J_-(z, {\bar y})=\emptyset$, for all $z\in\g_\gamma$ for $\gamma\in \D(\fu^+)\cup \D(\fm^+)$ and all ${\bar y}$ with maximal degree occurring in $u$. By Lemma~\ref{lem:tec.1}, this implies that all components $y_{i_j}$ of any ${\bar y}$ with maximal degree lie in a simple root space (if $y_{i_j}\in \g_{\alpha}$, for some non simple root $\alpha\in \dot{\D}(\fu^-)$, then Lemma~\ref{lem:tec.1} would imply the existence of $z\in\g_\gamma$ with $\gamma\in \D(\fn^+)$ such that $0\neq [z,y_\alpha]\in \fu^-$, which contradicts $J_-(z, {\bar y})=\emptyset$). In particular, this is true for ${\bar y}_{\max}=y_{a_1}(b_1)^{c_1}\cdots y_{a_s}(b_s)^{c_s}$. Thus, taking $e\ll 0$ and $1\leq l\leq s$, we get, by Lemma~\ref{lem:tec.2}
\begin{align*}
y_{-a_l}(e){\bar y}_{\max} x_{\max} & \equiv \sum_{\substack{1\leq l\leq s \\ a_j=a_l}} (-1)^{S_{1_j}S_{2_j}} \frac{c_j(c_j-1)}{2}\alpha_{a_l}(h_{a_l}) y_{a_j}(e+2b_j){\bar y}^{{\hat j}{\hat j}} x_{\max} \\
& + \sum_{\substack{1\leq l\leq s-1 \\ a_j=a_l}} (-1)^{S_{1_j}S_{2_j}S_{3_j}} \sum_{\xi	=j+1}^s c_j c_\xi \alpha_{a_\xi}(h_{a_l}) y_{a_\xi}(b_j+b_\xi+e){\bar y}^{{\hat j}{\hat \xi}} x_{\max} \\
& + \sum_{\substack{1\leq l\leq s \\ a_j=a_l}} (-1)^{S_{1_j}}  c_j {\bar y}_{\max}^{\hat j} h_{a_l} (b_j+e) x_{\max}\quad  \mod \bU(\cU^-)_{(c-2)}\otimes N,
\end{align*}

Let $S_1$, $S_2$ and $S_3$ be the first, second, and third summation above. Fix $l$, $l=1,\ldots, s$, and consider the summand
	\[
w={\bar y}_{\max}^{\hat l} h_{a_l} (b_l+e) x_{\max}
	\]
in the summation $S_3$. Write $h_{a_l}=z+h$, with $0\neq z\in \fh_X^\perp$ (since the roots occurring in ${\bar y}_{\max}$ are in $\dot{\D}(\fu)$) and $h\in \fh_X$. Since $z\neq 0$ and $e\ll 0$, it follows by Lemma~\ref{lem:tec.3} that $w\neq 0$. Also from Lemma~\ref{lem:tec.3} we get that $w$ does not lie in
	\[
\LinSpan\left( \sum_{\substack{1\leq l\leq s, \ j\neq l \\ a_j=a_l}} (-1)^{S_{1_j}}  c_j {\bar y}_{\max}^{\hat j} h_{a_l} (b_j+e) x_{\max}\right)+ \bU(\cU^-)_{(c-2)}\otimes N.
	\]
It is also clear that 
	\[
w\notin\LinSpan (S_1)+\LinSpan(S_2)+ \bU(\cU^-)_{(c-2)}\otimes N
	\]
as $S_1$ and $S_2$ have factors of the form $y_{a_j}(e+2b_j)$ or $y_{a_\xi}(b_j+b_\xi+e)$ on them while $w$ does not. 

Let us now consider some index $(\bi, \bm, \bp, j)\neq (\ba, \bb, \bc, d)$ and ${\bar y}=y_{\bi, \bm, \bp}$ and ${\bar x}=x_j$. Hence, Lemma~\ref{lem:tec.2} gives us that
\begin{align*}
y_{-a_l}(e){\bar y}x & \equiv  \sum_{j\in J_-} (-1)^{S_{1_j}} p_j[y_{-a_l}, y_{i_j}](m_j+e){\bar y}^{j}x \\
& +  \sum_{J^c} (-1)^{S_{1_j}}  p_j {\bar y}^{j} [y_{a_l}, y_{i_j}](m_j+e) x \\ 
& +  \sum_{1\leq j\leq r} (-1)^{S_{1_j}S_{2_j}} \frac{p_j(p_j-1)}{2}[y_{i_j},[y_{-a_l}, y_{i_j}]](2m_j+e){\bar y}^{{\hat j}{\hat j}}x \\
& + \sum_{j\in J_-^l\setminus \{1\}} (-1)^{S_{1_j}S_{2_j}S_{3_j}} \sum_{\eta=1}^{j-1} p_\eta p_j [y_{i_\eta},[y_{-a_l}, y_{i_j}]](e+m_j+m_\eta){\bar y}^{\hat{j}\hat{\eta}}x \\
& + \sum_{j\in J^c \setminus \{r\}} (-1)^{S_{1_j}S_{2_j}S_{3_j}} \sum_{\eta=j+1}^{r} [[y_{-a_l}, y_{i_j}],y_{i_\eta}] \\
& \cdot (e+m_j+m_\eta){\bar y}^{\hat{j}\hat{\eta}}x \quad \mod \bU(\cU^-)_{(p-2)}\otimes N,
\end{align*}
where $p=\deg {\bar y}$. Let $T_i$ with $1\leq i\leq 5$ be the summation of i-th line above. We claim that 
	\[
w\notin \sum_{1\leq i\leq 5}\LinSpan(T_i) + \bU(\cU^-)_{(p-1)}\otimes N.
	\]
Indeed, if $c'>p$, then $T_i$ for $2\leq i\leq 5$ have degree $p-1$ (and hence at most $c'-2$). In particular,
	\[
w\notin \sum_{2\leq i\leq 5}\LinSpan(T_i) + \bU(\cU^-)_{(p-1)}\otimes N.
	\]
On the other hand, $w\notin \LinSpan(T_1) + \bU(\cU^-)_{(p-1)}\otimes N$, by Lemma~\ref{lem:tec.3}. Thus the claim is proved if $c'>p$.

Finally assume that $p=c'$. Thus $y_{\bi,\bm,\bp}$ has maximal degree, and therefore, each  of its component $y_{i_j}$ must lie in a root space associated to a simple root. In particular, $J_-=\emptyset$, which implies $T_1=T_4=0$. Also, since $\alpha_{i_j}$ (here we are assuming that $y_{i_j}\in \g_{\alpha_{i_j}}$) is simple, we get that
	\[
[y_{-a_l}, y_{i_j}]=\delta_{i_j, a_l}h_{a_l},
	\]
which yields
	\[
[y_{i_\eta}, [y_{-a_l}, y_{i_j}]](e+m_j+m_\eta)=\delta_{i_j, a_l}\alpha_{i_\eta}(h_{a_l}) y_{i_\eta}.
	\]
Thus, $T_i$ for $i=3,4,5$, does not contain the factor $h_{a_l}(b_l+2)$ on it, and hence
	\[
w\notin \sum_{\substack{i=1 \\ i\neq 2}}^5\LinSpan(T_i) + \bU(\cU^-)_{(p-1)}\otimes N.
	\]
This shows that, in order for 
	\[
w\in \sum_{i=1}^5\LinSpan(T_i) + \bU(\cU^-)_{(p-1)}\otimes N,
	\]
we must have
	\[
{\bar y}^{\hat j}={\bar y}_{\max}^{\hat l}\text{ and } [y_{a_l}, y_{i_j}]=h_{a_l} \Rightarrow m_j=b_l,\quad  x=x_{\max}
	\]
which implies ${\bar y}={\bar y}_{\max}$ and therefore $(\bi, \bm, \bp, j)= (\ba, \bb, \bc, d)$. But this is a contradiction.
}
\end{proof}

The next theorem gives a criterion of simplicity for the generalized Verma type module $M_\cP(N)$ with a nonzero central charge, which  is the main result of this paper.

\begin{theo}\label{thm:main}
Suppose that $\lambda(K)\neq 0$, and let $N$ be a subquotient of $M_\cM(\lambda)$. Then $M_\cP(N)$ is a simple $\cG$-module if and only if $N$ is a simple $\cM$-module. 
\end{theo}
\begin{proof}
If $N'\subseteq N$ is a proper $\cM$-submodule of $N$, then it is standard to see that $M_\cP(N')$ is a proper $\cG$-submodule of $M_\cP(N)$.

Conversely, notice that the degree zero component of $M_\cP(N)\cong \bU(\cU^-)\otimes N$ with respect to the PBW filtration induced by $\bU(\cU^-)$ is given by $N$. In particular, $M_\cP(N)_{\lambda-\beta}\neq N_{\lambda-\beta}$ if and only if $\beta\in Q(\cN^+)\setminus Q(\cM^+)$.  Let $U$ be a submodule of $M_\cP(N)$. Then $U$ is a weight module, and $U\neq 0$ implies that $U_{\lambda-\beta}\neq 0$, for some $\beta\in Q(\cN^+)$. If $\beta\in Q(\cM^+)$, then 
	\[
U_{\lambda-\beta}\subseteq M_\cP(N)_{\lambda-\beta} = N.
	\] 
Since $N$ is simple, this implies that $U=M_\cP(N)$. On the other hand, if $\beta\in Q(\cN^+)\setminus Q(\cM^+)$, then one can apply Proposition~\ref{prop:inductive.argument} inductively to obtain $U_{\lambda-\gamma}\neq 0$, for some $\gamma\in Q(\cM^+)$. Thus we are back to the previous case and the result is proved.
\end{proof}

Theorem~\ref{thm:main} gives us some corollaries for the non-isotropic case.

\begin{cor}\label{cor-main}
Let $\cB$ be the Borel subalgebra associated to $P(\dot{\S}, X)$. Then the $\cG$-module $M_\cB(\lambda)$ is simple if and only if $\lambda(K)\neq 0$ and $M_\cK(\lambda)$ is a simple $\cK$-module.
\end{cor}
\begin{proof}
Recall that $\cB= \cH\oplus \cN^+ = (\cH\oplus \cM^+) \oplus  \cU^+$, and that $\cP=\cM\oplus \cU^+$. Then
	\[
M_\cP(M_\cM(\lambda)) = \Ind_{\cP}^{\cG} (\Ind_{(\cH\oplus \cM^+) \oplus  \cU^+}^\cM \C_\lambda) = \Ind_{\cB}^{\cG} \C_\lambda = M_\cB(\lambda).
	\]

Now recall that $\cM=\fH_X^-\oplus \cK\oplus\fH_X^+$ and $[\fH_X, \cK^\pm]=0$. Hence, if $\lambda(K)=0$, then, for any $h\in \fh_X^\perp$, we have that $h(-m)\cdot \cM_\cK(\lambda)$ generates a proper $\cM$-submodule of $M_\cM(\lambda)$. Thus $M_\cM(\lambda)$ is not a simple $\cM$-module and one implication follows from Theorem~\ref{thm:main}. Conversely, if $M_\cK(\lambda)$ is a simple $\cK$-module and $\lambda(K)\neq 0$, then it follows from Proposition~\ref{prop:inducing.from.K.to.M} that $M_\cM(\lambda)$ is a simple $\cM$-module. Thus the result follows from Theorem~\ref{thm:main}.
\end{proof}

In particular, for $X=\emptyset$ we immediately have

\begin{cor}\label{cor-empty}
Suppose that $X=\emptyset$. Then $M_\cB(\lambda)$ is simple if and only if $\lambda(K)\neq 0$.
\end{cor}
\begin{proof}
If $X=\emptyset$, then $\cM$ equals to the Heisenberg Lie algebra $\fH$. Now the result follows from Lemma~\ref{lem:irre.of.Heisenberg.modules}.
\end{proof}

\subsection{The cases: $\fsl(n,n)^{\widehat{}}$ and $\fpsl(n,n)^{\widehat{}}$}

In this section, we fix the following notation
	\[
\cG_{gl} = \fgl(n,n)^{\widehat{ }},\quad \cG_{sl} = \fsl(n,n)^{\widehat{ }}\quad \text{and}\quad \cG_{psl} = \fpsl(n,n)^{\widehat{ }}.
	\]
we also denote the subalgebras $\cH$, $\cB$, $\cK$, $\cM$, $\cU$, $\cP$ of $\cG_{gl}$ by $\cH_{gl}$, $\cB_{gl}$, $\cK_{gl}$, $\cM_{gl}$, $\cU_{gl}$, $\cP_{gl}$, respectively. In particular, these subalgebras induce (via intersection) subalgebras $\cH_{sl}$, $\cB_{sl}$, $\cK_{sl}$, $\cM_{sl}$, $\cU_{sl}$, $\cP_{sl}$ of $\cG_{sl}$. Notice that
\begin{gather*}
\cH_{gl} = \cH_{sl}\oplus \C e_{11},\quad \cB_{gl} = \cB_{sl}\oplus \C e_{11},\quad \cK_{gl} = \cK_{sl} \oplus \C e_{11}, \\
\cM_{gl} = \cM_{sl}\oplus L(\C e_{11})\quad \text{and}\quad  \cU_{gl} = \cU_{sl},
\end{gather*}
where $e_{11}$ denotes the elementary matrix with $1$ at the position $(1,1)$ and zeros elsewhere. Similarly, the subalgebras $\cH_{sl}$, $\cB_{sl}$, $\cK_{sl}$, $\cM_{sl}$, $\cU_{sl}$, $\cP_{sl}$ of $\cG_{sl}$ induce (via the canonical projection $\pi:\cG_{sl}\to \cG_{psl}$) subalgebras $\cH_{psl}$, $\cB_{psl}$, $\cK_{psl}$, $\cM_{psl}$, $\cU_{psl}$, $\cP_{psl}$ of $\cG_{psl}$. In particular, the only difference among these Lie subalgebras is their Cartan subalgebra.

Next we notice that any $\cG_{gl}$-module $M$ can be regarded as an $\cG_{sl}$-module via restriction. Conversely, let $M_{\cB_{sl}}(\lambda)$ (resp. $M_{\cM_{sl}}(F_{\cK_{sl}})$) be a Verma  (res. generalized standard Verma) type $\cG_{sl}$-module, where $F_{\cK_{sl}} = L_{\cK_{sl}}(\nu)$ is a subquotient of some standard Verma $\cK_{sl}$-module. Then we consider the Verma type module $M_{\cB_{gl}}(\lambda')$, where $\lambda'\in \cH_{gl}$ is any extension of $\lambda$, and the generalized standard Verma module $M_{\cM_{gl}}(F_{\cK_{gl}})$-module, where $F_{\cK_{gl}} = L_{\cK_{gl}}(\nu')$, is a subquotient of a standard Verma $\cK_{gl}$-module with $\nu'$ extending $\nu$.

\begin{theo}\label{thm:main.sl}
Suppose that $\lambda(K)\neq 0$, and let $N$ be a subquotient of $\cM_{\cM_{sl}}(\lambda)$. Then $M_{\cP_{sl}}(N)$ is a simple $\cG_{sl}$-module if and only if $N$ is a simple $\cM_{sl}$-module. 
\end{theo}
\begin{proof}
Let $N_{\cM_{sl}}$ be a simple subquotient of $\cM_{\cM_{sl}}(\lambda)$. Then, by Proposition~\ref{prop:inducing.from.K.to.M} (it can be checked that it still holds for $\cG_{sl}$), $N_{\cM_{sl}}\cong M_{\cM_{sl}}(F_{\cK_{sl}})$, where $F_{\cK_{sl}}$ is a simple subquotient of $\cM_{\cK_{sl}}(\lambda)$. It is clear that $F_{\cK_{gl}}$ is a simple $\cK_{gl}$-module. In particular, again by Proposition~\ref{prop:inducing.from.K.to.M}, $N_{\cM_{gl}}=M_{\cM_{gl}}(F_{\cK_{gl}})$ is simple, and hence, by Theorem~\ref{thm:main}, $M_{\cP_{gl}}(N_{\cM_{gl}})$ is a simple $\cG_{gl}$-module. Of course $M_{\cP_{sl}}(N_{\cM_{sl}})$ can be embedded into $M_{\cP_{gl}}(N_{\cM_{gl}})$. Then, for any element $m\in M_{\cP_{sl}}(N_{\cM_{sl}})\setminus N_{\cM_{sl}}$ there is $u\in\bU(\cM_{gl}^+\oplus \cU_{gl}^+)$ such that $0\neq um\in N_{\cM_{sl}}$. But notice that $\bU(\cM_{gl}^+\oplus \cU_{gl}^+) = \bU(\cM_{sl}^+\oplus \cU_{sl}^+)$. In other words, any $\cG_{sl}$-submodule of $M_{\cP_{sl}}(N_{\cM_{sl}})$ intersects $N_{\cM_{sl}}$ non-trivially. Since $N_{\cM_{sl}}$ is simple, it follows that $M_{\cP_{sl}}(N_{\cM_{sl}})$ is also simple, and the result follows.
\end{proof}

Any $\cG_{psl}$-module $W$ can be regarded as a $\cG_{sl}$-module through the canonical projection $\pi:\cG_{sl}\to \cG_{psl}$. We denote such a $\cG_{sl}$-module by $W^\pi$.

\begin{cor}
Suppose that $\lambda(K)\neq 0$, and let $N$ be a subquotient of $\cM_{\cM_{psl}}(\lambda)$. Then $M_{\cP_{psl}}(N)$ is a simple $\cG_{psl}$-module if and only if $N$ is a simple $\cM_{psl}$-module. 
\end{cor}
\begin{proof}
Let $N_{\cM_{psl}}$ be a simple subquotient of $\cM_{\cM_{psl}}(\lambda)$. Then $N_{\cM_{psl}}^\pi$ is a simple subquotient of $\cM_{\cM_{sl}}(\lambda)^\pi$. In particular, $M_{\cP_{sl}}(N_{\cM_{psl}}^\pi)$ is a simple $\cG_{sl}$-module by Theorem~\ref{thm:main.sl}. Since $M_{\cP_{sl}}(N_{\cM_{psl}}^\pi) = M_{\cP_{psl}}(N_{\cM_{psl}})^\pi$, this implies $M_{\cP_{psl}}(N_{\cM_{psl}})$ is simple, as we wanted.
\end{proof}

\section{Isotropic case}
Recall that the subalgebra $\cG^1$ of the Levi component $\cG^0$, which is given in Theorem~\ref{thm:descr.of.G_0}, may be of the form
	\[
\cG^1=\bigoplus_{i=1}^m L(\frak{g}^i)\oplus \C K,
	\]
where $[\frak{g}^i,\frak{g}^j]=0$ if $i\neq j$, and some $\frak{g}^i$ is isomorphic to the nilpotent Lie superalgebra $\frak{sl}(1,1)$.  In the previous sections we studied some facts about Verma and generalized Verma type modules that are associated to parabolic sets for which $X$ does not contain a connected component consisting of one isotropic odd root, that is, none of the subalgebras $\g^i$ of $\cG^1$ is isomorphic to $\frak{sl}(1,1)$. In this section we consider precisely the case that 
\begin{equation*}
X=\{\alpha\},\text{ where }\alpha\in \dot{\S}\text{ is isotropic (i.e. } (\alpha | \alpha)=0).
\end{equation*} 
In particular, we have $\cG^1=L(\frak{sl}(1,1))\oplus \C K$.

Let $\frak{h}_c$ be any subalgebra of $\fh$ such that $\fh=\C h_\alpha\oplus \fh_c$. (Since $\alpha$ is isotropic, we always have that $(h_\alpha | \fh_c)\neq 0$, see Example~\ref{ex:par.sets.sl(1,2)}). Hence, if we set
	\[
\fH_X := \C K\bigoplus_{k\in \Z \setminus \{0\}} (\frak{h}_c\otimes t^k),
	\]
then
	\[
\cG^0=\fH_X+\cG^1,\quad [\fH_X, \cG^1]\neq 0\quad\text{and}\quad \cG^1 \cap \fH_X=\C K.
	\]
Moreover, as in the previous sections, we consider
	\[
\cK:=\cG^1+\cH=(L(\frak{sl}(1,1)) \oplus \C K)\oplus \fh_c \oplus \C d
	\]
and 
	\[
\cM:=\cG^0 + \cH=\cK \oplus \fH_X,
	\]	
with standard triangular decompositions
	\[
\cK = \cK^-\oplus \cH\oplus \cK^+,\quad \text{where}\quad \cK^\pm = \fsl(1,1)\otimes t^{\pm 1}\C[t^{\pm 1}] \oplus \C x_\alpha,
	\]
and
	\[
\cM=\cM^+\oplus \cH\oplus \cM^+,\quad \text{where } \cM^\pm = \fH_X^\pm\oplus \cK^\pm.
	\]

The next result shows that Theorem~\ref{thm:main} also holds for the isotropic case.

\begin{theo}\label{thm:main.isot.case}
With the above notation, suppose that $\lambda(K)\neq 0$, and let $N$ be a subquotient of $M_\cM(\lambda)$. Then $M_\cP(N)$ is a simple $\cG$-module if and only if $N$ is a simple $\cM$-module. 
\end{theo}
\begin{proof}
The proof is similar to the non-isotropic case. Special care should be taken since we do not have an orthogonal decomposition for the Heisenberg subalgebra in the isotropic case.  
Assume that $M_\cP(N)$ is not simple, and let $W\subseteq M_\cP(N)$ be a proper submodule. Pick an arbitrary nonzero element $w\in W$. We can assume that $w$ is a weight element, and write
	\[
w = \sum_{i\in I} Z_i\otimes u_iv,
	\]
where $Z_i\in \bU(\cU^-)_{-\beta - l_i \alpha +k_i\delta}$ with  $\beta\in Q(\cU^-)$, $u_i\in \bU(\cM^-)$, $l_i, k_i$ integers, and $v\in N$ a highest weight vector.  We also assume that  all $Z_i$ are linearly independent and if  $\beta = \sum n_j\alpha_j $ is the decomposition of $\beta$ in simple roots then $\alpha_j\neq \alpha$ for all $j$ in this decomposition. 
Then set ${\rm ht}(w):={\rm ht}(\beta) = \sum n_j$.
 If  ${\rm ht}(\beta)>1$ then by Lemma~\ref{lem:tec.1} we find a nonzero weight element in $W$ of height ${\rm ht}(\beta)-1$. Hence we only need to consider the base of induction: ${\rm ht}(\beta)=1$. In this case $\beta$ is a simple root. If $l_i\neq l_j$ for some $i, j$ then choose a nonzero $x_{\alpha}\in \g_{\alpha}$ and  $e\gg 0$.  Since $x_{\alpha} (e)u_iv=0$ for all $i$, we get
 	\[
x_{\alpha} (e)w = \sum [x_{\alpha}, Z_i](e)\otimes u_iv.
	\]

Note that at least one of $[x_{\alpha}, Z_i](e)\in \bU(\cU^-)_{-\beta - (l_i-1) \alpha +(k_i+e)\delta}$ is not zero. Again, if coefficients by $\alpha$ in the remaining terms are not equal then choose $e'\gg e$ and apply  again $x_{\alpha} (e')$. After finitely many steps we will obtain a nonzero $w'\in W$,
	\[
w' = \sum_{j\in J} Z'_j\otimes u_jv
	\]
for some subset $J\subset I$, where $Z_j\in \bU(\cU^-)_{-\beta - m\alpha +r_j\delta}$ for some integer $m$, $r_j$, $j\in J$. Note that  the elements $u_j$, $j\in J$ are the same as in $w$.  Then for a nonzero 
  $x\in \g_{\beta+m\alpha}$ and $s\ll 0$, we get
	\[
x(s)w' = \sum_{j\in J} [x, Z'_i](s)\otimes u_j v = \sum_{j\in J} h(s+r_j)u_i v
	\]
for some $h\in \cH$ which is not multiple of $h_{\alpha}$. Note that $x(s)u_jv=0$ and  $s+r_j\ll 0$ by the choice of $s$   for all $j$. 	
We do not know at this point whether $x(s)w'\neq 0$.  Choose $j_0$ such that $r_{j_0}< r_j$ for all $j\in J$ and consider 
$$h(-s-r_{j_0})x(s)w' = \sum_{j\in J}h(-s-r_{j_0}) h(s+r_j)u_i v=\sum_{j\in J}[h(-s-r_{j_0}), h(s+r_j)]u_i v=\lambda(K) u_{j_0}v.$$ 
Since $\lambda(K)\neq 0$ we conclude that $u_{j_0}v\in W$. But   $u_{j_0}v\neq 0$, which gives a contradiction  since $W$ is proper and $N$ is simple.  The result now follows by induction.
\end{proof}

\begin{cor}\label{cor-main-iso-case}
Let $\cB$ be the Borel subalgebra associated to $P(\dot{\S}, X)$. Then the $\cG$-module $M_\cB(\lambda)$ is simple if and only if $\lambda(K)\neq 0$ and $M_\cM(\lambda)$ is a simple $\cM$-module.
\end{cor}
\begin{proof}
This follows from Theorem~\ref{thm:main.isot.case} along with the fact that
	\[
M_\cP(M_\cM(\lambda)) \cong M_\cB(\lambda).
	\]
\end{proof}

\subsection{Natural Verma $\fsl(1,1)^{\widehat{ }}$-modules}
In this section we study some facts about natural Verma modules over the Lie superalgebra $\cK=(L(\frak{sl}(1,1)) \oplus \C K)\oplus \fh_c \oplus \C d$.

Notice that the root system of $\cK$ is $P_0=\{\pm \alpha \pm n \delta,\ \pm k \delta\mid n\in \Z_{\geq 0},\ k\in \Z_{>0}\}$. Then, a natural partition of $P_0$ is given by
	\[
P_{0,\nat} := \{\alpha\pm n\delta,\ k\delta\mid n\in\Z_{\geq 0},\ k\in\Z_{>0}\}.
	\]
As usual, let 
	\[
\cB_\nat = \cH\oplus \bigoplus_{\gamma\in P_{0,\nat}} \cK_\gamma= \cH\oplus  L(\C x_\alpha)\oplus (h_\alpha\otimes t \C[t]).		\]
Fix $\zeta\in \cH^*$, and let $M_{\cB_{\nat}}(\zeta)$ be the Verma module of $\cK$ associated to $\zeta$ and $\cB_{\nat}$. By PBW Theorem, 
	\[
M_{\cB_{\nat}}(\zeta)_\mu\neq 0 \text{ then } \mu\in \{\zeta - k \alpha +\ell\delta,\ \zeta - n\delta \mid k\in \Z_{>0},\  \ell \in \Z,\ n\in \Z_{\geq 0}\}.
	\]

\begin{rem}
Unlike the Lie algebra case, and also the other cases of $\g$ that we considered in previous sections (see Corollary~\ref{cor-empty}), the next result shows that the natural Verma module $M_{\cB_\nat}(\zeta)$ is not simple for any $\zeta\in \cH^*$:
\end{rem}

\begin{prop}\label{prop:sl(1,1)-Verma-not-irred}
$M_{\cB_{\nat}}(\zeta)$ is not simple for any $\zeta\in \cH^*$.
\end{prop}
\begin{proof}
This follows from the fact that $\bU(h_\alpha\otimes t^{-1}\C[t^{-1}])$ is in the center of $L(\frak{sl}(1,1))$.
\end{proof}

Let $V(\zeta)$ denote the unique simple quotient of $M_{\cB_{\nat}}(\zeta)$. The next results contain relations that will be useful in this section.

\begin{lem}\label{lem:commut.relatios}
Let $m, \ell_1,\ldots, \ell_n\in \Z$ such that $\ell_i\neq \ell_j$ for all $i\neq j$. Then the following equation holds in $V(\zeta)$.
\begin{multline*}
x_\alpha(m)x_{-\alpha}(\ell_1)\cdots x_{-\alpha}(\ell_n)1_\zeta = \\ \sum_{i=1}^n (-1)^{i+1}(h_\alpha(m+\ell_i)+m\delta_{m, -\ell_i}(x_\alpha | x_{-\alpha})K) x_{-\alpha}(\ell_1)\cdots \widehat{x_{-\alpha}(\ell_i)}\cdots x_{-\alpha}(\ell_n)1_\zeta,
\end{multline*}
where the symbol $\widehat{x_{-\alpha}(\ell_i)}$ means that we are omitting this term from the summation. In particular,
If $m\neq -\ell_i$ for all $i=1,\ldots, n$, then 
	\[
x_\alpha(m)x_{-\alpha}(\ell_1)\cdots x_{-\alpha}(\ell_n)1_\zeta=0.
	\]
If $m=-\ell_j$ for a unique $j=1,\ldots, n$, then
\begin{align*}
x_\alpha(m)x_{-\alpha}(\ell_1) & \cdots x_{-\alpha}(\ell_n)1_\zeta \\
& =  (-1)^{j+1}(h_\alpha-\ell_j (x_\alpha | x_{-\alpha})K) x_{-\alpha}(\ell_1)\cdots \widehat{x_{-\alpha}(\ell_j)}\cdots x_{-\alpha}(\ell_n)1_\zeta \\
& = (-1)^{j+1}\zeta (h_\alpha-\ell_j (x_\alpha | x_{-\alpha})K) x_{-\alpha}(\ell_1)\cdots \widehat{x_{-\alpha}(\ell_j)}\cdots x_{-\alpha}(\ell_n)1_\zeta.
\end{align*}
\end{lem}

\begin{proof}
By the proof of Proposition~\ref{prop:sl(1,1)-Verma-not-irred}, we know that $h_\alpha(k)1_\zeta=0$, for every $k\in \Z\setminus\{0\}$. Now the result follows from the fact that $x_\alpha(m)1_\zeta=0$, for all $m\in \Z$, the elements of $L(\C h_\alpha)\oplus \C K$ commute with every element in $\g$, and the fact that
	\[
x_\alpha(m)x_{-\alpha}(\ell_i)= h_\alpha(m+\ell_i)+m\delta_{m, -\ell_i}(x_\alpha | x_{-\alpha})K - x_{-\alpha}(\ell_i) x_{\alpha}(m),
	\]
for every $i=1,\ldots, n$.
\end{proof}

\begin{cor}\label{lem:comm.relations.sl(1,1)}
Let $\ell_1,\ldots, \ell_n\in \Z$ such that $\ell_i\neq \ell_j$ for all $i\neq j$. Then the following equations hold in $V(\zeta)$.
\begin{enumerate}
\item \label{lem:comm.relations.sl(1,1)-1} Let $t\in \Z_{>0}$, and $m_1,\ldots, m_t\in \Z$. If there exists some $j=1,\ldots, t$, for which $m_j\neq -\ell_i$ for all $i=1,\ldots, n$, then 
	\[
x_\alpha(m_1)\cdots x_\alpha(m_t)x_{-\alpha}(\ell_1)\cdots x_{-\alpha}(\ell_n)1_\zeta=0.
	\]
\item \label{lem:comm.relations.sl(1,1)-3} If $t\leq n$, then 
\begin{multline*}
x_\alpha(-\ell_t)\cdots x_\alpha(-\ell_1)x_{-\alpha}(\ell_1)\cdots x_{-\alpha}(\ell_n)1_\zeta \\
= \prod_{i=1}^n \zeta (h_\alpha-\ell_i (x_\alpha | x_{-\alpha})K)x_{-\alpha}(\ell_{t+1})\cdots x_{-\alpha}(\ell_n) 1_\zeta.
\end{multline*}
In particular,
	\[
x_\alpha(-\ell_n)\cdots x_\alpha(-\ell_1)x_{-\alpha}(\ell_1)\cdots x_{-\alpha}(\ell_n)1_\zeta=\prod_{i=1}^n \zeta (h_\alpha-\ell_i (x_\alpha | x_{-\alpha})K) 1_\zeta.
	\]
\item \label{lem:comm.relations.sl(1,1)-2} If $t>n$, then 
	\[
x_\alpha(m_1)\cdots x_\alpha(m_t)x_{-\alpha}(\ell_1)\cdots x_{-\alpha}(\ell_n)1_\zeta=0,
	\]
for all $m_1,\ldots, m_t\in \Z$.
\end{enumerate}
\end{cor}

For every $\ell\in \Z$ and every $n\in \Z_{>0}$,  we define
	\[
\Z^n(\ell)=\{(\ell_1,\ldots, \ell_n)\in \Z^n\mid \ell_i\neq \ell_j,\  \ell_1+\cdots+\ell_n=\ell\}.
	\]
Next result characterizes the support of $V(\zeta)$.

\begin{prop}
Let $\ell\in \Z$ and $n\in \Z_{\geq 0}$. Then $V(\zeta)_{\zeta-n\alpha+\ell \delta}=0$ if and only if 		\[
\prod_{i=1}^n \zeta (h_\alpha-\ell_i (x_\alpha | x_{-\alpha})K) = 0,
	\]
for every $(\ell_1,\ldots, \ell_n)\in \Z^n(\ell)$.
\end{prop}
\begin{proof}
By the proof of Proposition~\ref{prop:sl(1,1)-Verma-not-irred}, $h_\alpha(-\ell)1_\zeta=0$, for every $\ell\in \Z_{>0}$. Then, by PBW Theorem,
	\[
V(\zeta)_{\zeta-n\alpha+\ell \delta}={\rm span}_{\C}\{x_{-\alpha}(\ell_1)\cdots x_{-\alpha}(\ell_n)1_\zeta\mid (\ell_1,\ldots, \ell_n)\in \Z^n(\ell)\}.
	\]
Since 
\begin{equation}\label{eq:weigh.space.null}
\begin{aligned}
(x_\alpha(m)+h_\alpha(k)) x_{-\alpha}(\ell)1_\zeta & = x_\alpha(m) x_{-\alpha}(\ell)1_\zeta \\
& = \zeta (h_\alpha(m+\ell)+m\delta_{m,-\ell} (x_\alpha | x_{-\alpha})K) 1_\zeta,
\end{aligned}
\end{equation}
for any $m\in \Z$, $k\in \Z_{>0}$, we obtain that $V(\zeta)_{\zeta-\alpha+\ell\delta}\neq 0$ if and only if $\zeta (h_\alpha-\ell (x_\alpha | x_{-\alpha})K) \neq 0$. We claim that $V(\zeta)_{\zeta-n\alpha+\ell\delta} \neq 0$ if and only if there exist an element $(\ell_1,\ldots, \ell_n)\in \Z^n(\ell)$ such that 
	\[
\prod_{i=1}^n \zeta (h_\alpha-\ell_i (x_\alpha | x_{-\alpha})K) \neq 0.
	\]
Indeed, it follows from Corollary~\ref{lem:comm.relations.sl(1,1)} that, if such a $(\ell_1,\ldots, \ell_n)\in \Z^n(\ell)$ does exist, then $x_{-\alpha}(\ell_1)\cdots x_{-\alpha}(\ell_n)1_\zeta\neq 0$, and $V(\zeta)_{\zeta-n\alpha+\ell \delta}\neq 0$.

Conversely, suppose that such a $(\ell_1,\ldots, \ell_n)\in \Z^n(\ell)$ does not exist. Then, for any $(\ell_1,\ldots, \ell_n)\in \Z^n(\ell)$ there exists $s\in \{1,\ldots, n\}$ such that $\zeta (h_\alpha-\ell_s (x_\alpha | x_{-\alpha})K) = 0$. In particular, by the $n=1$ case discussed above, we have
	\[
x_{-\alpha}(\ell_s)1_\zeta \in V(\zeta)_{\zeta-\alpha+\ell_s\delta}=0.
	\]
By PBW Theorem
	\[
\bU(L(\C x_{-\alpha})) \cong \Lambda (L(\C x_{-\alpha})),
	\]
where $\Lambda (L(\C x_{-\alpha}))$ denotes the exterior algebra associated to the loop space $L(\C x_{-\alpha})$.  Hence
	\[
x_{-\alpha}(\ell_1)\cdots x_{-\alpha}(\ell_n)1_\zeta  = -x_{-\alpha}(\ell_1)\cdots x_{-\alpha}(\ell_{s-1})x_{-\alpha}(\ell_n)x_{-\alpha}(\ell_{s+1})\cdots x_{-\alpha}(\ell_s)1_\zeta=0,
	\]
showing that $V(\zeta)_{\zeta-n\alpha+\ell\delta} = 0$. 
\end{proof}

\begin{cor}
Let $\ell\in \Z$. If $\zeta \neq 0$, then $V(\zeta)_{\zeta-n\alpha+\ell \delta}\neq 0$ for all $n\geq 2$.
\end{cor}

\begin{proof}
It follows from the fact that $| \ker (\zeta)\cap \{h_\alpha + \C K\}| = 1$ along with the fact that there exist two elements $(\ell_1,\ldots, \ell_n), (\ell_1',\ldots, \ell_n')$ in $\Z^n(\ell)$ such that $\ell_i\neq \ell_j'$ for all $i,j=1,\ldots, n$. 
\end{proof}
\begin{cor}
If $\zeta\neq 0$, then $V(\zeta)$ is infinite-dimensional. If $\zeta=0$, then $V(\zeta)=\C 1_\zeta$.
\end{cor}

\section*{Acknowledgment}
L.C. is supported in part by the Capes grant (88881.119190/2016-01) and by the PRPq grant ADRC-05/2016. V.F. is supported by the CNPq grant (304467/2017-0).  The second author is grateful to Ivan Dimitrov and Duncan Melville for many stimulating discussions.

\
\


\bibliographystyle{alpha}

\bibliography{/Users/lucascalixto/Dropbox/Research/Bib/bibliography.bib}

\newcommand{\arxiv}[1]{\href{http://arxiv.org/abs/#1}{\tt
  arXiv:\nolinkurl{#1}}}\newcommand{\doi}[1]{DOI:
  \href{http://dx.doi.org/#1}{\tt \nolinkurl{#1}}}
\begin{thebibliography}{DMP04}

\bibitem[Cox94]{Cox94}
B.~Cox.
\newblock Verma modules induced from nonstandard {B}orel subalgebras.
\newblock {\em Pacific J. Math.}, 165(2):269--294, 1994.

\bibitem[DFG09]{DFG09}
I.~Dimitrov, V.~Futorny, and D.~Grantcharov.
\newblock Parabolic sets of roots.
\newblock In {\em Groups, rings and group rings}, volume 499 of {\em Contemp.
  Math.}, pages 61--73. Amer. Math. Soc., Providence, RI, 2009.

\bibitem[DMP04]{DMP04}
I.~Dimitrov, O.~Mathieu, and I.~Penkov.
\newblock Errata to: ``{O}n the structure of weight modules'' [{T}rans. {A}mer.
  {M}ath. {S}oc. {\bf 352} (2000), no. 6, 2857--2869; mr1624174].
\newblock {\em Trans. Amer. Math. Soc.}, 356(8):3403--3404, 2004.

\bibitem[ERF09]{EF09}
S.~Eswara~Rao and V.~Futorny.
\newblock Integrable modules for affine {L}ie superalgebras.
\newblock {\em Trans. Amer. Math. Soc.}, 361(10):5435--5455, 2009.

\bibitem[FK02]{FK02}
D.~Fattori and V.~Kac.
\newblock Classification of finite simple {L}ie conformal superalgebras.
\newblock {\em J. Algebra}, 258(1):23--59, 2002.
\newblock Special issue in celebration of Claudio Procesi's 60th birthday.

\bibitem[FS93]{FS93}
V.~Futorny and H.~Saifi.
\newblock Modules of {V}erma type and new irreducible representations for
  affine {L}ie algebras.
\newblock In {\em Representations of algebras ({O}ttawa, {ON}, 1992)},
  volume~14 of {\em CMS Conf. Proc.}, pages 185--191. Amer. Math. Soc.,
  Providence, RI, 1993.

\bibitem[FSS00]{FSS00}
L.~Frappat, A.~Sciarrino, and P.~Sorba.
\newblock {\em Dictionary on {L}ie algebras and superalgebras}.
\newblock Academic Press, Inc., San Diego, CA, 2000.

\bibitem[Fut92]{Fut92}
V.~Futorny.
\newblock The parabolic sets of root system and corresponding representations
  of affine {L}ie algebras.
\newblock {\em Contemporary Math.}, 131(Part 2):45--52, 1992.

\bibitem[Fut94]{Fut94}
V.~Futorny.
\newblock Imaginary {V}erma modules for affine {L}ie algebras.
\newblock {\em Canad. Math. Bull.}, 37(2):213--218, 1994.

\bibitem[Fut97]{Fut97}
V.~Futorny.
\newblock {\em Representations of affine {L}ie algebras}, volume 106 of {\em
  Queen's Papers in Pure and Applied Mathematics}.
\newblock Queen's University, Kingston, ON, 1997.

\bibitem[JK89]{JK89}
P.~Jakobsen and V.~Kac.
\newblock A new class of unitarizable highest weight representations of
  infinite-dimensional {L}ie algebras. {II}.
\newblock {\em J. Funct. Anal.}, 82(1):69--90, 1989.

\bibitem[Kac78]{kac78}
V.~Kac.
\newblock Representations of classical {L}ie superalgebras.
\newblock In {\em Differential geometrical methods in mathematical physics,
  {II}, ({P}roc. {C}onf., {U}niv. {B}onn, {B}onn, 1977)}, volume 676 of {\em
  Lecture Notes in Math.}, pages 597--626. Springer, Berlin, 1978.

\bibitem[KW94]{KW94}
V.~Kac and M.~Wakimoto.
\newblock Integrable highest weight modules over affine superalgebras and
  number theory.
\newblock In {\em Lie theory and geometry}, volume 123 of {\em Progr. Math.},
  pages 415--456. Birkh\"auser Boston, Boston, MA, 1994.

\bibitem[KW01]{KW01}
V.~Kac and M.~Wakimoto.
\newblock Integrable highest weight modules over affine superalgebras and
  {A}ppell's function.
\newblock {\em Comm. Math. Phys.}, 215(3):631--682, 2001.

\bibitem[Ser11]{Ser11}
V.~Serganova.
\newblock Kac-{M}oody superalgebras and integrability.
\newblock In {\em Developments and trends in infinite-dimensional {L}ie
  theory}, volume 288 of {\em Progr. Math.}, pages 169--218. Birkh\"auser
  Boston, Inc., Boston, MA, 2011.
\newblock \doi{10.1007/978-0-8176-4741-4_6}.

\end{thebibliography}

\end{document}